\documentclass{book}
\usepackage{amsopn,amstext,amsbsy,amsmath,amscd,amsthm,amsfonts}
\usepackage[ma,mnf,bachelor]{mccover} 
\usepackage{geometry} 
\usepackage[latin9]{inputenc} 
\usepackage[all]{xy} 
\usepackage{indentfirst}
\usepackage{enumitem}
\usepackage{url}
\usepackage{listings}

\theoremstyle{plain}
\newtheorem{theorem}                 {\bf Theorem}      [chapter]

\newtheorem{proposition}  [theorem]  {Proposition}

\theoremstyle{definition}

\newtheorem{definition}   [theorem]  {Definition}

\newtheorem{remark}       [theorem]  {Remark}

\numberwithin{equation}{chapter}

\allowdisplaybreaks

\def \rn{{\mathbb R}}
\def \sn{{\mathbb S}}

\def \F{\mathcal F}

\def\nab#1#2{\hbox{$\nabla$\kern -.3em\lower 1.0 ex
		\hbox{$#1$}\kern -.1 em {$#2$}}}

\def \g{\mathfrak{g}}
\def \h{\mathfrak{h}}
\def \k{\mathfrak{k}}

\def \m{\mathfrak{m}}

\def \SL2{\widetilde{\text{\bf SL}}_{2}(\rn)}

\def \SO#1{\text{\bf SO}(#1)}

\def \U#1{\text{\bf U}(#1)}

\def \SU#1{\text{\bf SU}(#1)}
\def \su#1{\mathfrak{su}(#1)}

\DeclareMathOperator{\Div}{div}

\DeclareMathOperator{\grad}{grad}

\DeclareMathOperator{\trace}{trace}

\maintitle{Eight-Dimensional Hermitian Lie\\ Groups Conformally Foliated by\\ Minimal $\SU{2} \times \SU{2}$ Leaves}
\authors{Kexing Chen}
\issnum{2021}{14}
\sernum{LUNFMA}{4111}{2021}

\begin{document}


\large 

\frontcover 

\thispagestyle{empty}
\centerline {\bf\Large Abstract}
\vskip1cm

We investigate the $8$-dimensional Riemannian Lie groups $G^8$, carrying a left-invariant, conformal and minimal foliation $\F$, with leaves diffeomorphic to the subgroup $\SU 2\times\SU 2$ of $G^8$.  Such groups have been classified by E. Ghandour, S. Gudmundsson and T. Turner in their recent work \cite{Tur}.  They show that these $8$-dimensional Lie groups form a real  $13$-dimensional family.
\smallskip 

For each left-invariant Hermitian structure $J_\mathcal{V}$ on $\SU 2\times\SU 2$, we extend this to an almost Hermitian structures $J$ on $G$ adapted to the foliation $\F$ i.e. respecting the leaf structure on $G$ induced by $\F$.  We then classify those $8$-dimensional Lie groups $G$ for which the almost Hermitian structures $J$ are integrable ($\mathcal{W}_3\oplus\mathcal{W}_4$), semi-Kähler ($\mathcal{W}_4$), locally conformal Kähler ($\mathcal{W}_3$) or even Kähler ($\mathcal{K}$).
\smallskip

It turns out that for each $J_\mathcal{V}$ we obtain a $9$-dimensional family of Lie groups $G$ for which $J$ is integrable.  In the case of $J$ being semi-Kähler, we yield a $3$-dimensional family of such groups.  We then show that in the cases of $J$ being Kähler or locally conformal Kähler there are no solutions i.e. such $8$-dimensional Lie groups do not exist.

\vskip 10cm
{\it Throughout this work it has been my firm intention to give reference to the stated results and credit to the work of others. All theorems, propositions, lemmas and examples left unmarked are either assumed to be too well known or the fruits of my own efforts.}

\newpage 
\thispagestyle{empty}
\phantom{m}

\newpage 
\thispagestyle{empty}
\centerline {\bf\Large Acknowledgments}
\vskip1cm
I would like to express my gratitude to my supervisor Sigmundur Gudmundsson for his encouragement, patience and trust in me throughout this project. Without his passion and insight in differential geometry or the fruitful discussions between the two of us, this thesis would be impossible during such a difficult time of the pandemic. Besides, I would like to thank the friendly people at Lund University for their hospitality during my three years' study at Sweden. Best wishes to all of you,

\vskip1pc
\hskip9cm Kexing Chen
\phantom{m}

\newpage 
\thispagestyle{empty}

\newpage 
\tableofcontents
\thispagestyle{empty}
\phantom{m}

\newpage 
\setcounter{page}{0} 
\thispagestyle{empty}

\chapter{Riemannian Manifolds}

We assume that the readers of this thesis have some basic knowledge of differential geometry. They can refer to the textbook $\cite{Sig}$ by Sigmundur Gudmundsson for a good introduction. In this chapter, we introduce the main object that we investigate and cover some notions and results that will be used later in this thesis. All submanifolds are assumed to be smoothly embedded.

\section{Lie Groups}

A finite-dimensional \emph{Lie group} is a smooth manifold $G$ with a group structure such that the multiplication and the inversion are smooth. Let $G$ and $H$ be Lie groups, a map $\phi: G \to H$ is called a Lie group isomorphism if it is both a group homomorphism and a diffeomorphism. If such a map exists, $G$ and $H$ are said to be isomorphic. A subgroup $H$ of $G$ is called a Lie subgroup if $H$ is also a submanifold of $G$.\smallskip

The Lie algebra $\g$ of a Lie group $G$ is the vector space of left-invariant vector fields on $G$. This can be identified with the tangent space at the identity $e \in G$. An inner product structure on $\g$ induces a left-invariant Riemannian metric $g$ on $G$ by the left translations. A \emph{Riemannian Lie group} is a Lie group $G$ equipped with a left-invariant Riemannian metric.\smallskip

A \emph{Lie algebra} over a field $\mathbb{F}$ is a vector space $V$ over $\mathbb{F}$ endowed with a bilinear map $[,] : V \times V \to V$ satisfying
\begin{enumerate}[label=(\roman*)]
    \item \emph{antisymmetry}: $[X, Y] = -[Y, X]$ for any $X, Y \in V$,
    \item \emph{Jacobi identity}: $[X, [Y, Z]] + [Y, [Z, X]] + [Z, [X, Y]] = 0$ for any $X, Y, Z \in V$.
\end{enumerate}

Let $\g ,\h$ be two Lie algebras over a field $\mathbb{F}$, a linear map $\phi: \g \to \h$ is said to be a Lie algebra homomorphism if it preserves the Lie brackets i.e. $\phi([X, Y]) = [\phi(X), \phi(Y)]$. It is said to be a Lie algebra isomorphism if it is also bijective. The Lie algebras $\g$ and $\h$ are said to be isomorphic if there exists a Lie algebra isomorphism between them.\smallskip

The Lie algebra of a finite-dimensional Lie group is a finite-dimensional Lie algebra. Moreover, simply connected Lie groups can be classified by their Lie algebras. Specifically, we have the following result.

\begin{theorem}[{{\cite[page 531]{Lee}}}]
For any finite-dimensional Lie algebra $\g$ over $\mathbb{R}$, there exists a simply connected Lie group $G$ whose Lie algebra is isomorphic to $\g$. Furthermore, if $G$ and $H$ are simply connected Lie groups with isomorphic Lie algebras, then $G$ and $H$ are isomorphic.\label{Lie 1}
\end{theorem}

Another result about the correspondence between Lie subgroups and Lie subalgebras will be useful when we classify Lie groups containing $K$ as a Lie subgroup.

\begin{theorem}[{{\cite[page 506]{Lee}}}]
Let $G$ be a Lie group with Lie algebra $\g$. If $\k$ is a subalgebra of $\g$, then there exists a unique connected Lie subgroup $K$ of $G$ with Lie algebra $\k$.\label{Lie 2}
\end{theorem}

\section{Foliations}

Let $M$ be a smooth $m$-dimensional manifold with tangent bundle $TM$. A \emph{smooth distribution} $\mathcal{V}$ of rank $k$ is a $k$-dimensional subbundle of $TM$. In other words, it is the collection of some $k$-dimensional subspaces of the tangent spaces of $M$ that vary smoothly over $M$.\smallskip

A smooth distribution $\mathcal{V}$ on M is said to be \emph{involutive} if for any vector fields $X, Y$ in $\mathcal{V}$, the vector field $[X, Y]$ also belongs to $\mathcal{V}$. Given a smooth distribution $\mathcal{V}$ on $M$, a submanifold $N$ of $M$ is said to be an \emph{integral manifold} of $\mathcal{V}$ if $T_pN = \mathcal{V}_p$ for any $p \in M$. A connected integral manifold of $\mathcal{V}$ is said to be \emph{maximal} if it is not contained in any other connected integral manifold of $\mathcal{V}$.\smallskip

Next, we define the notion of a foliation on a smooth manifold $M$.

\begin{definition}\label{def of foliation}
Let $\mathcal{F} = \{L_\alpha\ | \alpha \in I\}$ be a collection of disjoint, connected, nonempty $k$-dimensional submanifolds of $M$. Then $\mathcal{F}$ is called a \emph{foliation} of dimension $k$ on $M$ if
\begin{enumerate}[label=(\roman*)]
    \item $\bigcup\limits_{\alpha \in I}L_\alpha = M$,
    \item for any point $p \in M$, there exists a chart $(U, \phi)$ about $p$ such that $\phi(U)$ is a cube in $\mathbb{R}^m$ and the intersection $L_\alpha \cap U$ for each $L_\alpha \in \mathcal{F}$ is either empty or a countable union of $k$-dimensional slices of the form $x^{k + 1} = c^{k + 1}, \dots, x^n = c^{m}$.  
\end{enumerate}
An element $L_\alpha$ of a foliation is called a \emph{leaf}. And (ii) is equivalent to the following condition: for each point $p \in M$, there is an open neighborhood $W$ of $p$ in $M$, a $(m-k)$-dimensional smooth manifold $N$ and a smooth submersion $\phi: W \to N$ such that for each $L_\alpha \in \mathcal{F}$, the connected components of $W \cap L_\alpha$ are the connected components of the fibres of $\phi$.
\end{definition}

Given a smooth submersion $\phi: M \to N$, by Definition \ref{def of foliation} and the implicit function theorem, we see that the collection $\mathcal{F}$ of the connected components of the fibres of $\phi$ is a smooth foliation on $M$. We will call $\mathcal{F}$ the \emph{foliation associated to} $\phi$. And a foliation is said to be \emph{simple} if it is associated to some global smooth submersion with connected fibres. The neighborhood $W$ in Definition \eqref{def of foliation} is called a \emph{$\mathcal{F}$-simple open set}.\smallskip

Foliations are related to involutive distributions by the following important global Frobenius theorem.

\begin{theorem}[{{\cite[page 502]{Lee}}}]\label{Frobenius}
Let $\mathcal{V}$ be an involutive distribution on a smooth manifold $M$. Then the collection of all maximal connected integral manifolds of $\mathcal{V}$ forms a foliation of $M$. Conversely, given a foliation $\mathcal{F}$ on $M$, the collection of tangent spaces of the leaves of $\mathcal{F}$ forms a distribution on $M$, which is called the associated distribution of $\mathcal{F}$. Furthermore, the associated distribution of a foliation $\mathcal{F}$ on $M$ is involutive.
\end{theorem}

A subgroup of a Lie group gives us examples of distributions and foliations that are of particular interest to us.

\begin{definition}
A distribution $\mathcal{V}$ on a Lie group $G$ is said to be \emph{left-invariant} if it is closed under left translations.
\end{definition}

\begin{definition}
    Let $G$ be a Lie group with Lie algebra $\g$. Then the distribution on $G$ generated by left translations of a subalgebra $\k$ of $\g$ is involutive. It is called the \emph{left-invariant distribution} generated by $\k$.
\end{definition}

By Theorem \eqref{Frobenius}, the left-invariant distribution generated by a Lie subalgebra $\k$ produces a foliation on $M$. We call it the \emph{left-invariant foliation} on $M$ generated by $\k$.

\section{Minimal Submanifolds}

In this section we introduce the mean curvature vector field and minimal submanifolds in preparation for the next section. Readers can find more details about this part in the textbook $\cite{Kob}$ by Kobayashi and Nomizu.\smallskip

Let $(N, h)$ be an $(m + r)$-dimensional Riemannian manifold and $M$ be an $m$-dimensional submanifold of $N$ equipped with the metric $g$ induced by $h$. The Levi-Civita connection on $N$ is denoted by $\nabla$. Let $X$ be a vector field on $M$ and $Z \in C^\infty(NM)$ be a vector field in the normal bundle $NM$ of $M$. We define a vector field by
\[
S_Z(X) = -(\nabla_XZ)^\top
\]
where $\top$ means the orthogonal projection onto $TM$.\smallskip

The operator $S$ generalizes the shape operator of a surface in $\mathbb{R}^3$ by the following \emph{Weingarten equation}.

\begin{proposition}[{{\cite[page 14]{Kob}}}]\label{shape operator}
The map $S: C^\infty(NM) \times C^\infty(TM) \to C^\infty(TM)$ is a tensor field on $M$. Therefore, for any point $p \in M$, the value of $(S_Z(X))_p$ only depends on the values of $Z$ and $X$ at $p$. Therefore, $S$ induces a well-defined bilinear map $S: T_xM^\perp \times T_xM \to T_xM$. Furthermore, we define the second fundamental form $B$ of $M$ by
\[
B(X, Y) = (\nabla_XY)^\perp
\]
where $X, Y \in T_pM$ and $\perp$ is the orthogonal projection onto $NM$. Then we have
\begin{equation}
    \langle S_Z(X), Y \rangle = \langle B(X, Y), Z \rangle \label{equation of shape operator}
\end{equation}
for any $X, Y \in T_pM$ and $Z \in T_pM^\perp$.
\end{proposition}
For a fixed $Z \in T_pM^\perp$, the map $X \mapsto S_Z(X)$ is a symmetric linear endomorphism of $T_xM$ by the symmetry of the second fundamental form and \eqref{equation of shape operator}. So we have a linear functional on $T_pM^\perp$ defined by
\[
Z \mapsto \frac{1}{m}\trace{S_Z}.
\]
From linear algebra we know that there exists a unique normal vector $\eta \in T_pM^\perp$ such that 
\[
g(\eta, Z) = \frac{1}{m}\trace{S_Z} \label{def of mc}
\]
for any $Z \in T_pM^\perp$. We call $\eta$ the \emph{mean curvature vector} of $M$ at $p$. We shall derive a formula for the mean curvature vector in an orthonormal basis. Let $\xi_1, ..., \xi_r$ be an orthonormal basis for $T_pM^\perp$ and $\beta_1, ..., \beta_m$ be an orthonormal basis for $T_pM$. The operator $S_{\xi_j}$ is abbreviated as $S_j$. Since the operator $S_Z$ is symmetric with respect to $g$, by \eqref{equation of shape operator} and \eqref{def of mc} we have
\begin{align}
    \eta &= \sum_{j = 1}^{r}\langle \eta, \xi_j \rangle \xi_j\nonumber\\
    &= \sum_{j = 1}^{r}(\frac{1}{m}\trace{S_j})\xi_j\nonumber\\
    &= \frac{1}{m}\sum_{j = 1}^{r}\sum_{k = 1}^{m}\langle S_j(\beta_k), \beta_k\rangle \xi_j\nonumber\\
    &= \frac{1}{m}\sum_{j = 1}^{r}\sum_{k = 1}^{m}\langle B(\beta_k, \beta_k), \xi_j\rangle \xi_j\nonumber\\
    &= \frac{1}{m}\sum_{j = 1}^{r}\Big\langle\sum_{k = 1}^{m} B(\beta_k, \beta_k), \xi_j\Big\rangle \xi_j\nonumber\\
    &= \frac{1}{m}\sum_{k = 1}^{m} B(\beta_k, \beta_k)\label{formula for eta}
\end{align}

Applying \eqref{formula for eta} to a local orthonormal frame of $TM$, we know that there is a well-defined smooth vector field $\eta^M \in C^\infty(NM)$ that takes the value of mean curvature vector of $M$ at every point. We call $\eta^M$ the \emph{mean curvature vector field} of $M$ in $(N, h)$.
\begin{definition}
    Let $(M, g)$ be a Riemannian manifold isometrically embedded into $(N, h)$. Then $M$ is said to be a \emph{minimal submanifold} if $\eta^M$ vanishes identically. By \eqref{formula for eta} we see that $M$ is minimal if and only if
    \[
    \eta^M = \trace{B} = \sum_{k = 1}^{m} B(X_k, X_k) = 0,
    \]
    for any local orthonormal frame $\{X_1, ..., X_m\}$ for $TM$.
\end{definition}

\section{Foliations on Riemannian Manifolds}

Now we can define some particular types of foliations on Riemannian manifolds. Readers are referred to \cite{Bai-Woo-book} for more details about this part.\smallskip

Let $(M, g)$ be a Riemannian manifold of dimension $m$ and denote the Levi-Civita connection of $M$ by $\nabla$. Let $\mathcal{V}$ be a $q$-dimensional distribution on $M$ and denote the orthogonal distribution $\mathcal{V}^\perp$ by $\mathcal{H}$. We call $\mathcal{V}$ the \emph{vertical distribution} and $\mathcal{H}$ the \emph{horizontal distribution}. The orthogonal projections from $TM$ onto $\mathcal{V}, \mathcal{H}$ are also denoted as $\mathcal{V}, \mathcal{H}$, respectively.\smallskip

Then we introduce second fundamental forms on distributions. More specifically we have the following definition.
\begin{definition}\label{second fundamental form}
    The \emph{second fundamental form} of the vertical distribution $\mathcal{V}$ is the tensor field $B^\mathcal{V}$ of type $(1, 2)$ on $M$ defined by
    \begin{equation}
        B^\mathcal{V}(X, Y) = \frac{1}{2}(\mathcal{H}(\nabla_{\mathcal{V}X}\mathcal{V}Y) + \mathcal{H}(\nabla_{\mathcal{V}Y}\mathcal{V}X)) \,\,\,\,\,\text{for any}\,\,\,\,\,X, Y \in C^\infty(TM).
    \end{equation}
Similarly, the second fundamental form of the horizontal distribution $\mathcal{H}$ is the tensor field $B^\mathcal{H}$ is defined by
    \begin{equation}
        B^\mathcal{H}(X, Y) = \frac{1}{2}(\mathcal{V}(\nabla_{\mathcal{H}X}\mathcal{H}Y) + \mathcal{V}(\nabla_{\mathcal{H}Y}\mathcal{H}X)) \,\,\,\,\,\text{for any}\,\,\,\,\,X, Y \in C^\infty(TM).
    \end{equation}
\end{definition}

Notice that if the distribution $\mathcal{V}$ is involutive, then for any point $p \in M$ and vectors $X, Y \in \mathcal{V}_p$, by Definition \ref{second fundamental form} and the fact the $\nabla$ is torsion-free, we have
\begin{align*}
    B^\mathcal{V}(X, Y) &= \frac{1}{2}(\mathcal{H}(\nabla_{\mathcal{V}X}\mathcal{V}Y) + \mathcal{H}(\nabla_{\mathcal{V}Y}\mathcal{V}X))\\
    &= \frac{1}{2}(\mathcal{H}(\nabla_XY) + \mathcal{H}(\nabla_YX))\\
    &= \frac{1}{2}(\mathcal{H}(\nabla_XY) + \mathcal{H}(\nabla_XY - [X, Y]))\\
    &= \mathcal{H}(\nabla_XY)
\end{align*}

So Definition \ref{second fundamental form} agrees with the usual notion of the second fundamental form on the integral manifolds of $\mathcal{V}$ when $\mathcal{V}$ is involutive.

\begin{definition}
    An involutive distribution $\mathcal{V}$ on a Riemannian manifold $(M, g)$ is said to be
    \begin{enumerate}[label=(\roman*)]
        \item \emph{minimal}, if the integral manifolds of $\mathcal{V}$ are all minimal i.e.
        \[
        \trace{B^{\mathcal{V}}} = \sum_{j = 1}^{m} B^{\mathcal{V}}(E_j, E_j) = 0,
        \]
        for some and hence for all orthonormal basis $\{E_1, \dots, E_m\}$ of $\mathcal{V}_p$ at any point $p$,
        \item \emph{totally geodesic}, if the integral manifolds of $\mathcal{V}$ are all totally geodesic i.e.
        \[
        B^\mathcal{V} \equiv 0,
        \]
        \item \emph{conformal}, if for vector fields $X, Y$ in the horizontal distribution $\mathcal{H}$, there exists a vector field $V$ in $\mathcal{V}$ such that
        \[
        B^\mathcal{H}(X, Y) = g(X, Y) \otimes V,
       \]
       \item \emph{Riemannian}, if $V \equiv 0$ in (iii) i.e. $B^\mathcal{H} \equiv 0$.
    \end{enumerate}
\end{definition}

A smooth foliation $\mathcal{F}$ on a Riemannian manifold $(M, g)$ is said to be \emph{minimal}, \emph{totally geodesic}, \emph{conformal}, or \emph{Riemannian} if the same applies to the associated vertical distribution $\mathcal{V}$.

\section{Harmonic Morphisms}

In this section, we give a brief introduction to harmonic morphisms and explain how foliations and harmonic morphisms are closely related. Finally we introduce the main object we investigate in this thesis.\smallskip

First we define the \emph{musical isomorphisms} between the tangent and cotangent bundles of a Riemannnian manifold.

\begin{definition}
    Let $(M, g)$ be a Riemannian manifold. For any vector field $X \in C^\infty(TM)$, there exists a unique 1-form $X^\flat$ called the \emph{flat} of $X$ such that
    \[
    X^\flat(Y) = g(X, Y),
    \]
    for any vector field $Y \in C^\infty(TM)$. Similarly, for any 1-form $\omega$ on $M$, there exists a unique vector field $\omega^\sharp$ called the \emph{sharp} of $\omega$ such that
    \[
    g(\omega^\sharp, Y) = \omega(Y),
    \]
    for any vector field $Y \in C^\infty(TM)$. The operators $\flat$ and $\sharp$ are mutually inverse vector bundle isomorphisms between $TM$ and $T^*M$.
\end{definition}

Let $X$ be a smooth vector field on a Riemannian manifold $(M, g)$. Then there is a $C^\infty(M)$-linear map from $C^\infty(TM) \times C^\infty(TM)$ to $C^\infty(M)$ defined by 
\begin{equation}
    (Y, Z) \mapsto g(\nabla_YX, Z). \label{nabla x}
\end{equation}
The map \eqref{nabla x} gives rise to a bilinear form on $T_pM$ for each point $p \in M$. We can take the trace of this map with respect to the metric tensor $g$. In particular, let $\{E_1, \dots, E_m\}$ be a local orthonormal frame, then we have
\[
\trace{\nabla X} = \sum_{j = 1}^m g(\nabla_{E_j}X, E_j).
\]

\begin{definition}\label{def of grad and div}
    Let $f: M \to \mathbb{R}$ be a smooth function on a Riemannian manifold $(M, g)$. Then the \emph{gradient} of $f$ is the vector field given by
    \begin{equation}
        \grad f = (df)^\sharp. \label{def of grad}
    \end{equation}
    Let $X$ be a smooth vector field on $M$. Then the \emph{divergence} of $X$ is defined by 
    \begin{equation}
        \Div X = \trace{\nabla X}. \label{def of div}
    \end{equation}
\end{definition}

By expanding \eqref{def of grad} and \eqref{def of div} by an orthonormal frame, we see that Definition \ref{def of grad and div} agrees with the usual definition of the gradient and divergence in $\mathbb{R}^n$. A smooth function $f: M \to \mathbb{R}$ on a Riemannian manifold $(M, g)$ is said to be \emph{harmonic} if 
\[
\Delta f = \Div \grad f \equiv 0
\]
and the operator $\Delta = \Div \grad$ is called the \emph{Laplace-Beltrami operator}.

Now we can define a harmonic morphism between Riemannian manifolds.
\begin{definition}\label{def of harmonic morphism}
    A smooth map $\phi: M \to N$ between Riemannian manifolds $(M, g)$ and $(N, h)$ is called a \emph{harmonic morphism} if for every harmonic function $f: V \to \mathbb{R}$ defined on an open subset $V$ of N such that $f^{-1}(V) \neq \emptyset$, the function $f \circ \phi$ is harmonic on $\phi^{-1}(V)$. 
\end{definition}

In this thesis we are particularly interested in conformal foliations of codimension 2 because of the following result, see page 128 in \cite{Bai-Woo-book}.
\begin{theorem}\label{conformal foliation produces harmonic morphisms}
Let $\mathcal{F}$ be a conformal foliation of codimension 2 on a Riemannian manifold $M$. Then $\mathcal{F}$ is minimal if and only if for any $\mathcal{F}$-simple open set $U$ of $\mathcal{F}$, there is a submersive harmonic morphism from $U$, to a 2-dimensional Riemannian manifold $N$, whose associated foliation is $\mathcal{F}|_U$.
\end{theorem}

So finding conformal and minimal foliations of codimension 2 will give us many interesting examples of harmonic morphisms. In the paper \cite{Tur}, the authors prove that for every 8-dimensional Lie group $G^8$ containing $K = \SU{2} \times \SU{2}$ as a Lie subgroup, the foliation $\F$ generated by $K$ is Riemannian and totally geodesic. Thus they obtain a 13-dimensional family of 8-dimensional Lie groups carrying a Riemannian foliation of codimension 2 with totally geodesic leaves. And this thesis is about investigating some structures from complex geometry on this family of Lie groups.

\chapter{Almost Hermitian Manifolds}

\section{Almost Hermitian Structures}

In this section, we introduce some interesting notions and results from complex geometry.\smallskip

By an n-dimensional \emph{complex manifold} we mean a second-countable Hausdorff space $M$ that admits an open cover $\{U_\alpha\}$ where each $U_\alpha$ is homeomorphic to some open subset of $\mathbb{C}^n$, which is called a \emph{complex chart}, and the transition maps are holomorphic.\smallskip

Let $M$ be a smooth manifold of real dimension $2n$. Then by an \emph{almost complex structure} on $M$ we mean a tensor field $J$ of type (1, 1) satisfying
\[
J^2 = -id.
\]
If an almost complex structure $J$ is defined on $M$, then $(M, J)$ is called an \emph{almost complex manifold}. A complex manifold of complex dimension $n$ has a canonical almost complex structure induced by its complex local charts. Let 
\[
z_k = x_k + i y_k, \,\,\,\,\,k = 1, \dots, n,
\]
be the complex coordinates for some complex chart. Then 
\[
x_k, y_k, \,\,\,\,\,k = 1, \dots, n,
\]
are local coordinates for the real manifold structure if we identify $\mathbb{C}^n$ with $\mathbb{R}^{2n}$. The map defined by
\[
\frac{\partial}{\partial x_k} \mapsto \frac{\partial}{\partial y_k}, \,\,\,\,\, \frac{\partial}{\partial y_k} \mapsto -\frac{\partial}{\partial x_k}, \,\,\,\,\,k = 1, \dots, n,
\]
is independent of the choice of the chart thus gives a canonical almost complex structure on $M$ induced by the complex one.\smallskip

Conversely, an almost complex structure $J$ on a smooth manifold $M$ is said to be \emph{integrable} if it is induced by some complex structure on $M$. The integrability problem can be converted into the calculation of some tensor fields by the following important Newlander-Nirenberg theorem.

\begin{theorem}[\cite{New-Nir}]\label{newlander nirenberg}
Let $J$ be an almost complex structure on a smooth manifold $M$. The Nijenhuis tensor $N_J$ of $J$ is the tensor field of type (1, 2) defined by
\begin{equation}
    N_J(X, Y) = [X, Y] + J[JX, Y] + J[X, JY] - [JX, JY], \label{nijenhuis}
\end{equation}
for all vector fields $X, Y$ on $M$. Then $J$ is integrable if and only if $N_J \equiv 0$.
\end{theorem}

Now let $(M,g)$ be a $2n$-dimensional Riemannian manifold. An almost complex structure $J$ is said to be \emph{compatible} with $g$ if
\[
g(X, Y) = g(JX, JY),
\]
for any point $p \in M$ and $X, Y \in T_pM$. A Riemannian manifold $(M, g)$ with a compatible almost complex structure $J$ is called an \emph{almost Hermitian manifold}. If this almost complex structure $J$ is integrable, then $g$ induces a \emph{Hermitian metric} on the complex structure induced by $J$ and we call $(M, g, J)$ a \emph{Hermitian manifold}.\smallskip

In this thesis we are particularly interested in the left-invariant almost Hermitian structures on Riemannian Lie groups.

\begin{definition}
    Let $G$ be a $2n$-dimensional Lie group. Let $J_e$ be a complex structure on its Lie algebra $\g$. Then the \emph{left-invariant almost complex structure} $J$ on $G$ induced by $J_e$ is defined by
    \begin{equation}
        J(X) = dL_p \circ J_e \circ dL_{p^{-1}}(X), \label{left-invariant complex structure}
    \end{equation}
    for any $p \in G$, $X \in T_pG$, where $L_p$ denotes the left translation on $G$ by $p$.
\end{definition}

\begin{definition}
     A \emph{Hermitian Lie group} is a Riemannian Lie group $(G, g)$ equipped with a left-invariant complex structure $J$ that is compatible with $g$. So in this case $G$ is both a Lie group and a Hermitian manifold.
\end{definition}

Given a left-invariant almost Hermitian structure $J$ on a Riemannian Lie group $(G, g)$, the integrability problem of $J$ can be simplified by the following well-known result.

\begin{proposition}
Let $J$ be a left-invariant almost complex structure on a $2n$-dimensional Lie group $G$. Then the Nijenhuis tensor $N_J$ vanishes if and only if $N_J$ is zero at the tangent space $T_eG\cong\g$. 
\end{proposition}
\begin{proof}
Let $p$ be an element of $G$ and $X_p, Y_p$ be arbitrary vectors in $T_pG$. Then there exist left-invariant vectors fields $X, Y$ with the values $X_p, Y_p$ at $p$, respectively. Therefore
\begin{align*}
    N_J(X_p, Y_p) &= N_J(X, Y)(p)\\
    &= [X_p, Y_p] + J[JX_p, Y_p] + J[X_p, JY_p] - [JX_p, JY_p].
\end{align*}
Since the vector fields $X, Y$ are left-invariant, we have
\begin{align*}
    J[JX_p, Y_p] &= dL_p \circ J \circ dL_{p^{-1}}([dL_p \circ J \circ dL_{p^{-1}}(X_p), Y_p])\\
    &= dL_p \circ J ([J \circ dL_{p^{-1}}(X_p), dL_{p^{-1}}(Y_p)])\\
    &= dL_p(J[J X_e, Y_e])
\end{align*}
and similarly for the other terms that
\begin{align*}
    [X_p, Y_p] &= dL_p([X_e, Y_e]),\\
    J[X_p, JY_p] &= dL_p(J[X_e, JY_e]),\\
    [JX_p, JY_p] &= dL_p([JX_e, JY_e]).
\end{align*}
Therefore we have
\[
N_J(X_p, Y_p) = dL_p(N_J(X_e, Y_e))
\]
which accomplishes the proof.

\end{proof}

\section{A Classification of Almost Hermitian Structures}

In this section, we present the famous classification of almost Hermitian structures from the paper \cite{Gra}, where the authors A. Gray and L. M. Hervella classify them into 16 different types.\smallskip

For an almost Hermitian manifold $(M^{2n}, g, J)$, the K\"{a}hler form $\omega$ is the differential 2-form defined by
\begin{equation}
    \omega = g(JY, Z), \,\,\,\,\,\,\,\,\,\, Y, Z \in C^\infty(TM).\label{kaehler form}
\end{equation}
Denote the Levi-Civita connection of $M$ by $\nabla$. Then there is a well-defined 3-covariant tensor field $\nabla\omega$ on $M$ defined by
\begin{equation}
    (\nabla_X\omega)(Y, Z) = X(\omega(Y, Z)) - \omega(\nabla_XY, Z) - \omega(Y, \nabla_XZ),\label{nabla omega}
\end{equation}
where $X, Y, Z \in C^\infty(TM)$.\smallskip

\begin{proposition}\label{symmetry of nabla}
Let $(M^{2n}, g, J)$ be an almost Hermitian manifold, then the K\"{a}hler form $\omega$ satisfies
\begin{enumerate}[label=(\roman*)]
    \item $(\nabla_X\omega)(Y, Z) = -(\nabla_X \omega )(Z, Y),$
    \item $(\nabla_X\omega)(Y, Z) = -(\nabla_X\omega)(JY, JZ),$
\end{enumerate}
for any vector fields $X, Y, Z$ on $M$.
\end{proposition}
\begin{proof}
The first equality follows directly from \eqref{kaehler form} and \eqref{nabla omega}. Now we consider (ii). Since the Levi-Civita connection is compatible with the metric $g$, we have
\begin{align*}
    (\nabla_X\omega)(Y, JY) &= X(\omega(Y, JY)) - \omega(\nabla_XY, JY) - \omega(Y, \nabla_XJY)\\
    &= X(g(JY, JY)) - g(J\nabla_XY, JY) - g(JY, \nabla_XJY)\\
    &= X(g(JY, JY)) - g(\nabla_XY, Y) - g(JY, \nabla_XJY)\\
    &= X(g(JY, JY)) - \frac{1}{2}X(g(Y, Y)) - \frac{1}{2}X(g(JY, JY))\\
    &= 0.
\end{align*}
Then the equality (ii) follows by expanding $(\nabla_X\omega)(Y + Z, Y + Z) = 0$.

\end{proof}

Let $V$ be a $2n$-dimensional real vector space with an inner product $\langle,\rangle$ and an almost complex structure $J$ such that $\langle X,Y \rangle = \langle JX,JY \rangle$, for any $X, Y \in V$. Let $V^*$ be the dual space of $V$. By a direct calculation we see that $W$ defined by 
\begin{align}
    W = \{\alpha \in \otimes^3 V^*|&\alpha(X, Y, Z) = -\alpha(X, Z, Y),\nonumber\\  &\alpha(X, Y, Z) = -\alpha(X, JY, JZ) \, \text{ for all } X, Y, Z \in V\}.\label{def of w}
\end{align}
is a linear subspace of $\otimes^3 V^*$. Then there is a natural inner product on $W$ defined by
\[
\langle \langle \alpha, \beta \rangle \rangle = \sum_{j, k, l = 1}^{2n}\alpha(E_j, E_k, E_l)\beta(E_j, E_k, E_l),
\]
for any orthonormal basis $\{E_1, ..., E_{2n}\}$ of $V$. Following the notations in \cite{Gra}, for $\alpha \in W$, we define $\Tilde{\alpha}\in V^*$ by
\begin{equation}
    \Tilde{\alpha}(X) = -\sum_{j = 1}^{2n}\alpha(E_i, E_i, X), \label{def of tilde alpha}
\end{equation}
where $\{E_1, \dots, E_{2n}\}$ is an arbitrary orthonormal basis for $V$. We also define the following four subspaces of $W$ by
\begin{align*}
    W_1 = \{\alpha \in W |\, &\alpha(X, X, Z) = 0 \, \text{ for all } X, Z \in V\},\\
    W_2 = \{\alpha \in W |\, &\alpha(X, Y, Z) + \alpha(Z, X, Y) + \alpha(Y, Z, X) = 0 \, \text{ for all } X, Y, Z \in V\},\\
    W_3 = \{\alpha \in W |\, &\alpha(X, Y, Z) - \alpha(JX, JY, Z) = 0, \Tilde{\alpha}(Z) = 0 \, \text{ for all } X, Y, Z \in V\},\\
    W_4 = \{\alpha \in W |\, &\alpha(X, Y, Z) = -\frac{1}{2(n - 1)}\big(\langle X, Y\rangle \Tilde{\alpha}(Z)-\langle X, Z\rangle\Tilde{\alpha}(Y)\\&-\langle X, JY\rangle\Tilde{\alpha}(JZ)+\langle X, JZ\rangle\Tilde{\alpha}(JY)\big) \, \text{ for all } X, Y, Z \in V\}.
\end{align*}

Since $J$ is compatible with the inner product of $V$, we can equip it with a natural structure of a Hermitian vector space of complex dimension $n$ by setting
\[
(a + bi)X = aX + bJX,
\]
where $a, b \in \mathbb{R}$ and $X \in V$. Hence, if we fix an orthonormal basis $\{E_1, ..., E_n\}$ for $V$, there is a natural representation $\rho$ of the unitary group $\U{n}$ on $\otimes^3 V^*$ defined by
\[
\rho(A)(\phi)(X, Y, Z) = \phi(AX, AY, AZ),
\]
where $A \in \U{n}$, $\phi \in \otimes^3 V^*$ and $X, Y, Z \in V$.\smallskip

By Definition \ref{def of w}, we see that $W$ is invariant under the action of the unitary group $\U{n}$. So $\rho$ induces a representation of $\U{n}$ on the vector space $W$. The following result is vital to the classification.

\begin{theorem}[\cite{Gra}]\label{decomposition of W}
The vector space $W$ has the following orthogonal decomposition
\[
W = W_1 \oplus W_2 \oplus W_3 \oplus W_4,
\]
where each component $W_j$ is invariant under the action of \U{n}. Furthermore, the induced representation of $\U{n}$ on each $W_j$ is irreducible.
\end{theorem}

Let $(M^{2n}, g, J)$ be an almost Hermitian manifold and $p$ be an element of $M$. By comparing Proposition \ref{symmetry of nabla} with \eqref{def of w}, we can put $V = T_pM$, then $(\nabla\omega)|_p \in W$. We denote the four irreducible components of this $W$ by $W_{p, j}$, $j = 1, 2, 3, 4$.\smallskip

Now we can classify all of the almost Hermitian manifolds according to $\nabla\omega$.
\begin{definition}
The class of almost Hermitian manifolds $(M^{2n}, g, J)$ whose K\"{a}hler form $\omega$ satisfies
\[
(\nabla\omega)|_p \in W_{p, j},
\]
for each point $p \in M$ is denoted by $\mathcal{W}_j$, where $j = 1, 2, 3, 4$. And $W_{p, j}$ can be replaced by $W_{p, j} \oplus W_{p, k}$ when the class is denoted as $\mathcal{W}_j \oplus \mathcal{W}_k$, similarly for the other direct sums. In particular, we denote $\mathcal{K} = \bigcap\limits_{j = 1}^4 \mathcal{W}_j$.
\end{definition}

Hence, we have 16 classes of different almost Hermitian structures. In the paper \cite{Gra}, Gray and Hervella prove that each class defined in this way coincides with some well-known class of almost Hermitian manifolds. We give an brief account of their proofs for the case $\mathcal{W}_3 \oplus \mathcal{W}_4$, $\mathcal{W}_3$, $\mathcal{W}_4$ and $\mathcal{K}$.
\begin{definition}
An almost Hermitian manifold $(M, g, J)$ is said to be
\begin{enumerate}[label=(\roman*)]
    \item \emph{semi-K\"{a}hler} if $\delta\omega = 0$,
    \item \emph{locally conformal K\"{a}hler} if the tensor field $\mu = 0$,
    \item \emph{K\"{a}hler} if $J$ is integrable and $d\omega = 0$.
\end{enumerate}
Here, the operator $\delta$ in (i) is the codifferential operator given by
\begin{equation}
    \delta: \Omega^{k + 1}(M) \to \Omega^{k}(M),\;\;\;\;
    \delta = - \sum_{j = 1}^{2n}E_j \mathbin{\lrcorner} \nabla_{E_j},\label{codifferential}
\end{equation}
see page 101 of \cite{Mor}. The (1, 2) tensor field $\mu$ in (ii) is defined by
\begin{align}
&\langle\mu(X, Y), Z \rangle\nonumber\\
= \;&(\nabla_X\omega)(Y, Z) + \frac{1}{2(n - 1)}\{\langle X, Y\rangle\cdot\delta\omega(Z) - \langle X, Z\rangle\cdot\delta\omega(Y)\nonumber\\ &\qquad \qquad - \langle X, JY\rangle\cdot\delta\omega(JZ) + \langle X,JZ\rangle\cdot\delta\omega(JY)\},
\end{align}
for any $X, Y, Z \in C^\infty(TM)$. It is a locally conformal invariant of almost Hermitian manifolds and vanishes when $M$ is K\"{a}hler (see Lemma 4.1 in \cite{Gra}). Besides, it can be verified that the condition $d\omega = 0$ in (iii) is equivalent to $\nabla\omega = 0$.
\end{definition}

\begin{theorem}[\cite{Gra}]
An almost Hermitian manifold $(M, g, J)$ belongs to $\mathcal{W}_3 \oplus \mathcal{W}_4$, $\mathcal{W}_3$, $\mathcal{W}_4$ or $\mathcal{K}$ if and only if $M$ is Hermitian, Hermitian semi-K\"{a}hler, locally conformal K\"{a}hler or K\"{a}hler, respectively.
\end{theorem}
\begin{proof}
First we show that
\[
W_3 \oplus W_4 = \{\alpha \in W |\, \alpha(X, Y, Z) - \alpha(JX, JY, Z) = 0 \,\text{ for all } X, Y, Z \in V\}.
\]
If $\alpha \in W$ satisfies $\alpha(X, Y, Z) - \alpha(JX, JY, Z) = 0$, we define $\alpha_1, \alpha_2$ by
\begin{align*}
    \alpha_2(X, Y, Z) = -\frac{1}{2(n - 1)}&\big\{\langle X, Y\rangle \Tilde{\alpha}(Z)-\langle X, Z\rangle\Tilde{\alpha}(Y)\\&-\langle X, JY\rangle\Tilde{\alpha}(JZ)+\langle X, JZ\rangle\Tilde{\alpha}(JY)\big\},\\
    \alpha_1(X, Y, Z) = \alpha(X, Y, Z) &- \alpha_2(X, Y, Z).
\end{align*}
By a direct calculation, we have 
\[
\Tilde{\alpha_2}(Z) = \Tilde{\alpha}(Z).
\]
Hence, $\alpha = \alpha_1 + \alpha_2$ is a well-defined decomposition of $\alpha$ in $W_3 \oplus W_4$. The inclusion in the opposite direction can be verified in straightforward way.\smallskip

Gray and Hervella prove the following equalities between $N_J$ and $\omega$.
\[
\langle N_J(X, Y), JZ\rangle = \langle (\nabla_XJ)(Y) - (\nabla_{JX}J)(JY) - (\nabla_YJ)(X) + (\nabla_{JY}J)(JX), Z \rangle,
\]
\begin{align*}	
&2\langle (\nabla_XJ)(Y) - (\nabla_{JX}J)(JY) , Z\rangle\\ = \;&\langle N_J(X, Y), JZ\rangle - \langle N_J(Y, Z), JX\rangle + \langle N_J(Z, X), JY\rangle.
\end{align*}
Since
\begin{align*}
    \langle (\nabla_XJ)(Y), Z\rangle &= \langle \nabla_XJY - J\nabla_XY, Z \rangle\\
    &= (\nabla_X\omega)(Y, Z),
\end{align*}
we have
\[
\langle (\nabla_XJ)(Y) - (\nabla_{JX}J)(JY), Z\rangle = (\nabla_X\omega)(Y, Z) - (\nabla_{JX}\omega)(JY, Z).
\]
Combining these results together, we see that $M$ is Hermitian if and only if
\[
(\nabla_X\omega)(Y, Z) - (\nabla_{JX}\omega)(JY, Z) = 0.
\]
So $\mathcal{W}_3 \oplus \mathcal{W}_4$ is the class of Hermitian manifolds.\smallskip

By comparing \eqref{codifferential} with the definition of $\Tilde{\alpha}$, we see that a manifold belongs to $\mathcal{W}_3$ if and only if $\delta\omega = 0$ i.e. $\mathcal{W}_3$ is the class of Hermitian semi-K\"{a}hler manifolds. Similarly, a manifold belongs to $\mathcal{W}_4$ if and only if $\mu = 0$ i.e. $\mathcal{W}_4$ is the class of locally conformal K\"{a}hler manifolds. Finally, $\mathcal{K}$ is the class of K\"{a}hler manifolds since it contains exactly the manifolds satisfying $\nabla\omega = 0$.

\end{proof}

\chapter{Hermitian Foliations}

In this chapter, we give a detailed account of the 13-dimensional family of Lie groups $G^8$ introduced in Chapter 1. We define a left-invariant Hermitian structure $J$ on $G^8$ and investigate the conditions for $J$ to be integrable.
\section{Left-invariant Complex Structures on $\su{2} \oplus \su{2}$}

The left-invariant complex structures on $\su{2} \oplus \su{2}$ have been studied by L. Magnin in his paper \cite{Mag}. Specifically, he proves the following result. 
\begin{proposition}
The set of integrable left-invariant almost complex structures on $\su2 \oplus \su2$ consists of matrices of the form
\begin{equation}
\normalsize{
\left(\begin{array}{@{}cccccc@{}}
	\lambda_1^2\xi & -\lambda_1\mu_1\xi+\nu_1 & \lambda_1\nu_1\xi+\mu_1 & \eta\lambda_1\lambda_2 & -\eta\lambda_1\mu_2 & \eta\lambda_1\nu_2 \\
	-\lambda_1\mu_1\xi-\nu_1 & \mu_1^2\xi & \lambda_1-\mu_1\nu_1\xi & -\eta\mu_1\lambda_2 & \eta\mu_1\mu_2 & -\eta\mu_1\nu_2 \\
	\lambda_1\nu_1\xi-\mu_1 & -\lambda_1-\mu_1\nu_1\xi & \nu_1^2\xi & \eta\nu_1\lambda_2 & -\eta\nu_1\mu_2 & \eta\nu_1\nu_2 \\
	-\frac{\xi^2+1}{\eta}\lambda_1\lambda_2 & \frac{\xi^2+1}{\eta}\mu_1\lambda_2 & -\frac{\xi^2+1}{\eta}\nu_1\lambda_2 & -\lambda_2^2\xi & \lambda_2\mu_2\xi + \nu_2 & -\lambda_2\nu_2\xi + \mu_2 \\
	\frac{\xi^2+1}{\eta}\lambda_1\mu_2 & -\frac{\xi^2+1}{\eta}\mu_1\mu_2 & \frac{\xi^2+1}{\eta}\nu_1\mu_2 & \lambda_2\mu_2\xi - \nu_2 & -\mu_2^2\xi & \lambda_2 + \mu_2\nu_2\xi \\
	-\frac{\xi^2+1}{\eta}\lambda_1\nu_2 & \frac{\xi^2+1}{\eta}\mu_1\nu_2 & -\frac{\xi^2+1}{\eta}\nu_1\nu_2 & -\lambda_2\nu_2\xi-\mu_2 & -\lambda_2 + \mu_2\nu_2\xi & -\nu_2^2\xi 
\end{array}\right),
}\label{matrix JV}
\end{equation}
where
\[
(\xi, \eta) \in \rn \times \rn^*,\;\;
\left(\begin{array}{@{}c@{}}
    \lambda_j \\
    \mu_j \\
    \nu_j\\
  \end{array}\right) \in \sn ^2, \;\; j = 1, 2.
\]
Furthermore, these matrices form a closed 6-dimensional smooth submanifold of $\rn^{36}$ diffeomorphic to $\rn \times \rn^* \times \sn^2 \times \sn^2$.\label{prop 2.2}
\end{proposition}

Magnin notices that the matrix $J_\mathcal{V}$ in Proposition \ref{prop 2.2} can be written in a different way. Recall the Hopf fibration $p: \mathbb{S}^3 \to \mathbb{S}^2$ defined by
\[
p(x, y, z, w) = \big(2(w y - x z), 2(w x + y z), 2x^2 + 2y^2 - 1\big).
\]
Hence, for $(\lambda_j, \mu_j, \nu_j) \in \mathbb{S}^2$, $j = 1, 2$, there exists $(x_j, y_j, z_j, w_j) \in \mathbb{S}^3$, $j = 1, 2$, such that
\begin{align*}
    \lambda_j &= 2(w_j y_j - x_j z_j),\\
    \mu_j &= 2(w_j x_j + y_j z_j),\\
    \nu_j &= 2x_j^2 + 2y_j^2 - 1.
\end{align*}
There is a double cover of $\SO{3}$ by $\SU{2} \cong \mathbb{S}^3$ that maps $q = (x_j, y_j, z_j, w_j) \in \mathbb{S}^3$ to
\[M_j = 
\begin{pmatrix}
   x_j^2 - y_j^2 - z_j^2 + w_j^2 & -2(x_jy_j + z_jw_j) & 2(-x_jz_j + w_jy_j)\\
   2(-w_jz_j + x_jy_j)  & x_j^2 - y_j^2 + z_j^2 - w_j^2 & -2(w_jx_j + y_jz_j)\\
   2(w_jy_j + x_jz_j) & 2(w_jx_j - y_jz_j) & x_j^2 + y_j^2 - z_j^2 - w_j^2
\end{pmatrix} \in \SO{3},
\]
$j = 1, 2$. By a direct calculation, we see that the matrix $J_\mathcal{V}$ in Proposition \ref{prop 2.2} can be written as
\begin{equation}
    J_\mathcal{V} = \Phi \Psi(\xi, \eta) \Phi^{-1}\label{simpler J}
\end{equation}
where
\[
\Phi = \left(\begin{array}{@{}cc@{}}
    M_{1} & 0 \\\
    0 & M_{2}
  \end{array}\right)
\]
and
\[
\Psi(\xi, \eta) = 
  \left(\begin{array}{@{}cccccc@{}}
    0 & 1 & 0 & 0 & 0 & 0 \\
    -1 & 0 & 0 & 0 & 0 & 0 \\
    0 & 0 & \xi & 0 & 0 & \eta \\
    0 & 0 & 0 & 0 & 1 & 0 \\
    0 & 0 & 0 & -1 & 0 & 0 \\
    0 & 0 & -\frac{1+\xi^2}{\eta} & 0 & 0 & -\xi
  \end{array}\right).
\]
Conversely, we can verify that any matrix of the form \eqref{simpler J} can be written as \eqref{matrix JV} by reversing the process above.
\section{Hermitian Lie Groups}

Let $G^8$ be a simply connected 8-dimensional Lie group containing $K = \SU{2} \times \SU{2}$ as a Lie subgroup. Denote the Lie algebras of $G^8$ and $K$ by $\g$ and $\k$, respectively. By Theorem \ref{Lie 1} and \ref{Lie 2}, the Lie group structure of $G^8$ is totally determined by the Lie algebra $\g$ containing $\k = \su{2} \oplus \su{2}$ as a Lie subalgebra upto isomorphism. The standard structure equations of $\k$ are given by
\begin{align*}
    [A, B] = 2C,\;\;   
    [B, C] &= 2A,\;\;  
    [C, A] = 2B,\\
    [R, S] = 2T,\;\;\;  
    [S, T] &= 2R,\;\; 
    [T, R] = 2S,
\end{align*}
where $\k$ is equipped with the standard metric induced by its Killing form and $\{A, B, C, R, S, T\}$ is the standard orthonormal basis for $\k$ i.e. each copy of $\su{2}$ has the orthonormal basis $\{A, B, C\}$, $\{X, Y, Z\}$, respectively. We can extend $\{A, B, C, R, $ $S, T\}$ to a basis for $\g$ by adding vectors $X$ and $Y$. By declaring $$\mathcal{B} = \{A, B, C, R, S, T, X, Y\}$$ to be an orthonormal basis for $\g$, we can define a left-invariant metric on $G^8$. The orthogonal complement of $\k$ in $\g$, which is denoted by $\m$, has the orthonormal basis $\{X, Y\}$. The Lie algebra $\g$ is well-defined if and only if the Lie brackets relations among basis vectors is antisymmetric and satisfy the Jacobi identity introduced in Chapter 1. The problem of classifying such Lie algebras has been completely solved by the following result, see Proposition 4.1 in \cite{Tur}.

\begin{proposition}
Let $G$ be an 8-dimensional Riemannian Lie group containing $K = \SU{2} \times \SU{2}$ as a Lie subgroup, where $K$ is endowed with the standard Riemannian metric induced by its Killing form. Then the Lie bracket relations for $\g$ take the following form
\begin{align}
    [A, X] = -b_{11}B -c_{11}C,\;\;  &[A, Y] = -b_{21}B -c_{21}C, \nonumber\\   
    [B, X] = b_{11}A -c_{12}C,\;\;   &[B, Y] = b_{21}A -c_{22}C, \nonumber\\    
    [C, X] = c_{11}A +c_{12}B,\;\;   &[C, Y] = c_{21}A +c_{22}B, \nonumber\\    
    [R, X] = -s_{14}S -t_{14}T,\;\;  &[R, Y] = -s_{24}S -t_{24}T, \label{eq1}\\    
    [S, X] = s_{14}R -t_{15}T,\;\;  &[S, Y] = s_{24}R -t_{25}T, \nonumber\\     
    [T, X] = t_{14}R +t_{15}S,\;\;  &[T, Y] = t_{24}R +t_{25}S, \nonumber\\
    [X, Y] = \rho X + \theta_1 A + \theta_2 B &+ \theta_3 C + \theta_4 R + \theta_5 S + \theta_6 T, \nonumber
\end{align}
where 
\[
\left(\begin{array}{@{}c@{}}
\theta_1\\
\theta_2\\
\theta_3\\
\theta_4\\
\theta_5\\
\theta_6
  \end{array}\right)
=\frac{1}{2}
\left(\begin{array}{@{}c@{}}
-\rho c_{12} + b_{11}c_{21} - b_{21}c_{11}\\
\rho c_{11} + b_{11}c_{22} - b_{21}c_{12}\\
-\rho b_{11} + c_{11}c_{22} - c_{12}c_{21}\\
-\rho t_{15} + s_{14}t_{24} - s_{24}t_{14}\\
\rho t_{14} + s_{14}t_{25} - s_{24}t_{15}\\
-\rho s_{14} + t_{14}t_{25} - t_{15}t_{24}
  \end{array}\right)
\]
and the 13 coefficients $b_{jk}, c_{jk}, s_{jk}, t_{jk}, \rho$ are arbitrary real numbers.\label{Turner equations}
\end{proposition}
Let $\mathcal{V}, \mathcal{H}$ be the left-invariant distributions generated by $\k, \m$, respectively. Then $TM$ has an orthogonal decomposition $TM = \mathcal{V} \oplus \mathcal{H}$. Let $J$ be a left-invariant almost Hermitian structure on $G^8$ such that $J(\mathcal{V}) = \mathcal{V}$ and $J(\mathcal{H}) = \mathcal{H}$. Then we can assume that $J(X) = Y$ and $J(Y) = -X$, otherwise we interchange $X, Y$ and the resulting equations only differ by a sign. Therefore, we can write $J$ as a $8 \times 8$ block matrix
$$J = \begin{pmatrix}
J_{\mathcal{V}} & 0\\
   0 & J_{\mathcal{H}}
\end{pmatrix},\ \ \text{where}\ \ J_{\mathcal{H}}=
\begin{pmatrix}
0 & -1\\
1 &  0
\end{pmatrix}\in\rn^{2\times 2}
$$
and $J_{\mathcal{V}}\in \rn^{6 \times 6}$ satisfies $J_{\mathcal{V}}^2=-I_6$. So $J_{\mathcal{V}}$ is a complex structure on $\k$. Since $J(\mathcal{V}) \subseteq \mathcal{V}$ and $[\mathcal{V}, \mathcal{V}] \subseteq \mathcal{V}$, we can restrict $N_J$ to the Nijenhuis tensor $N_{J_{\mathcal{V}}}$ on $\mathcal{V}$. Hence, $N_J = 0$ implies that $N_{J_{\mathcal{V}}} = 0$ i.e. $J_{\mathcal{V}}$ must be an integrable left-invariant almost complex structure on $\su{2} \oplus \su{2}$, which is given by \eqref{matrix JV}.\smallskip

The structure $J$ preserves the metric i.e. $J \in \SO{8}$ if and only if $J_{\mathcal{V}} \in \SO{6}$, which is equivalent to $\Psi(\xi, \eta) \in \SO{6}$ by \eqref{simpler J}. So it suffices to solve the following system of equations:
\begin{align*}
    \xi^2 + \eta^2 &= 1,\\
    (-\frac{1+\xi^2}{\eta})^2 + (-\xi)^2 &= 1,\\
    -\frac{1+\xi^2}{\eta}\xi - \xi \eta &= 0. 
\end{align*}
The only solution to the system is $\xi = 0, \eta = \pm 1$.\smallskip

When $\eta = 1$ we have
\begin{equation}
\Psi(\xi, \eta) = 
  \left(\begin{array}{@{}cccccc@{}}
    0 & 1 & 0 & 0 & 0 & 0 \\
    -1 & 0 & 0 & 0 & 0 & 0 \\
    0 & 0 & 0 & 0 & 0 & 1 \\
    0 & 0 & 0 & 0 & 1 & 0 \\
    0 & 0 & 0 & -1 & 0 & 0 \\
    0 & 0 & -1 & 0 & 0 & 0
  \end{array}\right).\nonumber
\end{equation}
And when $\eta = -1$ we have
\begin{equation}
\Psi(\xi, \eta) = 
  \left(\begin{array}{@{}cccccc@{}}
    0 & 1 & 0 & 0 & 0 & 0 \\
    -1 & 0 & 0 & 0 & 0 & 0 \\
    0 & 0 & 0 & 0 & 0 & -1 \\
    0 & 0 & 0 & 0 & 1 & 0 \\
    0 & 0 & 0 & -1 & 0 & 0 \\
    0 & 0 & 1 & 0 & 0 & 0
  \end{array}\right).\nonumber
\end{equation}

When $\eta = 1$, consider the action of the almost complex structure $J$ on $\k = \su2 \oplus \su2$ geometrically. Firstly, $J$ rotates the two copies of $\su2$ by the corresponding matrix $M_1^T, M_2^T \in \SO{3}$, respectively. Then it gives the three mutually orthogonal subspaces spanned by $\{A, B\}$, $\{C, T\}$, $\{R, S\}$ a structure of complex plane respectively, i.e. the two copies of $\su2$ are interconnected by the basis vector $C$ and $T$. Finally, the two copies are rotated back according to $M_1, M_2 \in \SO{3}$ respectively.\smallskip

In this thesis, we solve the integrability conditions of $J$ for the case $\eta = 1$. The other cases when $\eta = -1$ can be solved in a similar way. When $\xi = 0, \eta = 1$, the matrix $J$ becomes
\begin{equation}
  \left(\begin{array}{@{}cccccccc@{}}
    0 & \nu_1 & \mu_1 & \lambda_1\lambda_2 & -\lambda_1\mu_2 & \lambda_1\nu_2&0&0 \\
   -\nu_1 & 0 & \lambda_1 & -\mu_1\lambda_2 & \mu_1\mu_2 & -\mu_1\nu_2 &0&0\\
    -\mu_1 & -\lambda_1 & 0 & \nu_1\lambda_2 & -\nu_1\mu_2 & \nu_1\nu_2 &0&0\\
    -\lambda_1\lambda_2 & \mu_1\lambda_2 & -\nu_1\lambda_2 & 0 & \nu_2 & \mu_2 &0&0\\
  \lambda_1\mu_2 & -\mu_1\mu_2 & \nu_1\mu_2 & - \nu_2 & 0 & \lambda_2 &0&0\\
    -\lambda_1\nu_2 & \mu_1\nu_2 & -\nu_1\nu_2 & -\mu_2 & -\lambda_2 & 0&0&0\\
    0&0&0&0&0&0&0&-1\\
     0&0&0&0&0&0&1&0
  \end{array}\right).\label{matrixJ}
\end{equation}
For convenience, the basis $\{A, B, C, R, S, T, X, Y\}$ is denoted by $\{e_1, \dots e_8\}$. We use the two notations interchangeably.\smallskip

Now we can state the main result of this thesis.
\begin{theorem}
Let $G^8$ be a 8-dimensional Riemannian Lie group containing $K = \SU{2} \times \SU{2}$ as a Lie subgroup, where $K$ is endowed with the standard Riemannian metric induced by its Killing form. The Lie bracket relations are given by Proposition \ref{Turner equations}. Let $J$ be the left-invariant almost Hermitian structure on $G^8$ defined by $\eqref{matrixJ}$, where $\lambda_j^2 + \mu_j^2 + \nu_j^2 = 1$, $j = 1, 2$. Without loss of generality, we assume that $\lambda_1, \lambda_2 \neq 0$. Then $J$ is integrable if and only if the coefficients of $\g$ satisfy the following relations
\begin{align}
    c_{11} &= \frac{\lambda_1^2+\mu_1^2}{\lambda_1}b_{21}+\frac{\mu_1}{\lambda_1}c_{12}-\frac{\mu_1\nu_1}{\lambda_1}c_{21}-\nu_1c_{22},\nonumber\\
    b_{11} &= \frac{\mu_1\nu_1}{\lambda_1}b_{21} + \frac{\nu_1}{\lambda_1}c_{12} -\frac{\lambda_1^2+\nu_1^2}{\lambda_1}c_{21} + \mu_1c_{22},\nonumber\\
    t_{14} &= \frac{\lambda_2^2+\mu_2^2}{\lambda_2}s_{24}+\frac{\mu_2}{\lambda_2}t_{15}-\frac{\mu_2\nu_2}{\lambda_2}t_{24}-\nu_2t_{25},\nonumber\\
    s_{14} &= \frac{\mu_2\nu_2}{\lambda_2}s_{24} + \frac{\nu_2}{\lambda_2}t_{15} -\frac{\lambda_2^2+\nu_2^2}{\lambda_2}t_{24} + \mu_2t_{25},\nonumber
\end{align}
where $b_{21}, c_{12}, c_{21}, c_{22}, s_{24}, t_{15}, t_{24}, t_{25}$ are arbitrary real numbers.
\label{integrability theorem} 
\end{theorem}
\begin{remark}
For each case, the dimension of the solutions to the system is 8. Notice that in the Lie bracket relations of $\g$, the coefficient $\rho$ is an arbitrary real number. Therefore, for any left-invariant almost Hermitian structure $J$ on $G^8$ given by \eqref{matrix JV}, there exists exactly a 9-dimensional family of coefficients for the Lie brackets relations of $\g$ such that $(G^8, g, J)$ is a Hermitian Lie group. The space of solutions to the problem therefore forms a vector bundle of rank 9 over $\mathbb{S}^2 \times \mathbb{S}^2$. 
\end{remark}
\begin{proof}
We denote $\langle N_J(e_j, e_\alpha), e_k \rangle$ by $N_J(e_j, e_\alpha)_k$ and the dimension of $G^8$ by $2n$. Since $J_{\mathcal{V}}$ is integrable and $N_J$ restricts to $N_{J_{\mathcal{V}}}$, we have
\[
N_J(e_j, e_k) = 0,
\]
where $j, k \in \{1,\dots, 6\}$. By the formula of $J$, we have
\begin{align*}
    N_J(X, Y) &= [X, Y] + J[JX, Y] + J[X, JY] - [JX, JY]\\
              &= [X, Y] + J[Y, Y] + J[X, -X] - [Y, -X]\\
              &=0.
\end{align*}
Since the Nijenhuis tensor $N_J$ is antisymmetric, for $\alpha, \beta \in \{7, 8\}$, we have
\[
N_J(e_\alpha, e_\beta) = 0.
\]
Now we consider the case when $j \in \{1, \dots, 6\}$ and $\alpha \in \{7, 8\}$. Since $J(\mathcal{V})\subseteq\mathcal{V}$ and $\mathcal{H}([E, F]) = 0$ for $E \in \mathcal{V}, F \in \mathcal{H}$, we have
\begin{align*}
N_J(e_j, e_\alpha)_\beta &= \big\langle [e_j, e_\alpha] + J[Je_j, e_\alpha] + J[e_j, Je_\alpha] - [Je_j, Je_\alpha], e_\beta\big\rangle\\
&= 0,
\end{align*}
where $\beta \in \{7, 8\}$. So it suffices to check
\begin{equation}
N_J(e_j, e_\alpha)_k = 0,\nonumber
\end{equation}
where $j, k \in \{1,\dots, 6\}$ and $\alpha \in \{7, 8\}$.\smallskip

There are $6 \times 6 \times 2 = 72$ equations to solve. To simplify the system of equations, first of all we prove the following equality
\begin{equation}
N_J(e_j, e_\alpha)_k + N_J(e_k, e_\alpha)_j = 0\label{lemma0}
\end{equation}
for any $j, k \in \{1, 2, 3, 4, 5, 6\}$ and $\alpha \in \{7, 8\}$. By definition, we have
\[
N_J(e_j, e_\alpha)_k = \big\langle [e_j, e_\alpha] + J[Je_j, e_\alpha] + J[e_j, Je_\alpha] - [Je_j, Je_\alpha], e_k\big\rangle,
\]
\[
N_J(e_k, e_\alpha)_j = \big\langle [e_k, e_\alpha] + J[Je_k, e_\alpha] + J[e_k, Je_\alpha] - [Je_k, Je_\alpha], e_j\big\rangle.
\]
From the Lie bracket relations \eqref{eq1} for $G^8$ we can directly see that
\begin{equation}
\big\langle [e_j, e_\alpha], e_k \big\rangle = -\big\langle [e_k, e_\alpha], e_j \big\rangle.\label{lemma1}
\end{equation}
Next we show that the sum of the third item of $N_J(e_j, e_\alpha)_k$ and the fourth item of $N_J(e_k, e_\alpha)_j$ is zero. By definition, we have
\begin{align}
    &\big\langle J[e_j, Je_\alpha], e_k \big\rangle - \big\langle [Je_k, Je_\alpha], e_j \big\rangle\nonumber\\= &-\Big\{\big\langle [e_j, Je_\alpha], Je_k \big\rangle + \big\langle [Je_k, Je_\alpha], e_j \big\rangle \Big\}\nonumber\\
    = &-\Big\{ \Big\langle \sum_{l = 1}^{2n}\langle [e_j, Je_\alpha], e_l\rangle e_l, \sum_{l=1}^{2n}\langle Je_k, e_l \rangle e_l\Big\rangle\nonumber\\  &+\Big\langle [\sum_{l=1}^{2n}\langle Je_k, e_l \rangle e_l, Je_\alpha], e_j \Big\rangle \Big\}\nonumber\\
    =&-\Big\{ \sum_{l=1}^{2n}\big\langle [e_j, Je_\alpha], e_l \big\rangle \big\langle Je_k, e_l\big\rangle\nonumber\\ &+ \sum_{l=1}^{2n}\big\langle Je_k, e_l \big\rangle \big\langle [e_l, Je_\alpha], e_j \big\rangle \Big\}\nonumber\\
    = &-\sum_{l = 1}^{2n - 2}\big\langle Je_k, e_l \big\rangle \Big( \big\langle [e_j, Je_\alpha], e_l \big\rangle + \big\langle [e_l, Je_\alpha], e_j \big\rangle \Big) \nonumber\\
    &- \sum_{l = 2n - 1}^{2n}\big\langle Je_k, e_l \big\rangle \Big( \big\langle [e_j, Je_\alpha], e_l \big\rangle + \big\langle [e_l, Je_\alpha], e_j \big\rangle \Big) \nonumber\\
    = &- \sum_{l = 1}^{2n - 2}\big\langle Je_k, e_l \big\rangle \Big( \big\langle [e_j, Je_\alpha], e_l \big\rangle + \big\langle [e_l, Je_\alpha], e_j \big\rangle \Big) \nonumber\\
    =&\;0\label{lemma2},
\end{align}
where the second last equality holds since $\mathcal{H}(Je_k) = 0$ by \eqref{eq1}, and the last equality follows from $\eqref{lemma1}$. By the symmetry in $j, k$, we see that the sum of the fourth item of $N_J(e_j, e_\alpha)_k$ and the third item of $N_J(e_k, e_\alpha)_j$ is also zero.\smallskip

For the second item in $N_J(e_j, e_\alpha)_k$, we have
\begin{align*}
    \Big\langle J[Je_j, e_\alpha], e_k \Big\rangle &=-\Big\langle [Je_j, e_\alpha], Je_k\Big\rangle\\
    &=-\Big\langle \big[ \sum_{l=1}^{2n}\langle Je_j, e_l\rangle e_l, e_\alpha \big], \sum_{r=1}^{2n}\langle Je_k, e_r\rangle e_r \Big\rangle\\
    &= -\Big\langle \sum_{r=1}^{2n}\big\langle\big[ \sum_{l=1}^{2n}\langle Je_j, e_l\rangle e_l, e_\alpha \big], e_r\big\rangle e_r, \sum_{r=1}^{2n}\langle Je_k, e_r\rangle e_r \Big\rangle\\
    &= -\sum_{r=1}^{2n}\big\langle \big[ \sum_{l=1}^{2n}\langle Je_j, e_l\rangle e_l, e_\alpha \big], e_r \big\rangle \big\langle Je_k, e_r\big\rangle\\
    &= -\sum_{r=1}^{2n}\sum_{l=1}^{2n}\langle Je_j, e_l\rangle\langle [e_l, e_\alpha], e_r\rangle \langle Je_k, e_r\rangle.
\end{align*}
Switching $k$ and $j$, by the symmetry in $j, k$, we see that
\begin{align*}
    \Big\langle J[Je_k, e_\alpha], e_j \Big\rangle &= -\sum_{r=1}^{2n}\sum_{l=1}^{2n}\langle Je_k, e_l\rangle\langle [e_l, e_\alpha], e_r\rangle \langle Je_j, e_r\rangle\\
    &=-\sum_{r=1}^{2n}\sum_{l=1}^{2n}\langle Je_j, e_l\rangle\langle [e_r, e_\alpha], e_l\rangle \langle Je_k, e_r\rangle,
\end{align*}
where the last equality follows by renaming the indices $l, r$. Adding the two equalities together, by \eqref{lemma1}, we have
\begin{align}
	&\Big\langle J[Je_j, e_\alpha], e_k \Big\rangle + \Big\langle J[Je_k, e_\alpha], e_j \Big\rangle\nonumber\\ = &-\sum_{r=1}^{2n}\sum_{l=1}^{2n}\langle Je_j, e_l\rangle \langle Je_k, e_r\rangle (\langle [e_l, e_\alpha], e_r\rangle + \langle [e_r, e_\alpha], e_l\rangle)\nonumber\\
	= &-\sum_{r=1}^{2n - 2}\sum_{l = 1}^{2n - 2}\langle Je_j, e_l\rangle \langle Je_k, e_r\rangle (\langle [e_l, e_\alpha], e_r\rangle + \langle [e_r, e_\alpha], e_l\rangle)\nonumber\\
	 = &\;0\label{lemma3},
\end{align}
where the second equality holds since $\mathcal{H}(Je_j) = \mathcal{H}(Je_k) = 0$ by \eqref{eq1}.\smallskip

The proof of $\eqref{lemma0}$ is accomplished by adding \eqref{lemma1}, \eqref{lemma2} and \eqref{lemma3}. Therefore, it suffices to solve the system of $2\times{6\choose 2} = 30$ equations given by
\begin{equation}
N_J(e_j, e_\alpha)_k = 0,\nonumber
\end{equation}
where $j, k \in \{1, 2, 3, 4, 5, 6\}$, $j<k$, and $\alpha \in \{7, 8\}$. These 30 equations can be divided into 2 cases
\begin{enumerate}[label= (\roman*)]
    \item $j \in \{1, 2, 3\}$, $k \in \{4, 5, 6\}$, $\alpha \in \{7, 8\}$;
    \item $j, k \in \{1, 2, 3\}$ or $j, k \in \{4, 5, 6\}$, $\alpha \in \{7, 8\}$.
\end{enumerate}

Case (i) consists of $3 \times 3 \times 2 = 18$ equations, and case (ii) has $2 \times 2 \times {3 \choose 2} = 12$ equations. By a standard calculation, we see that the two cases satisfy the following linear relations
\begin{align}
    N_J(e_1, e_\alpha)_4 &= \lambda_2 N_J(e_2, e_\alpha)_3 + \lambda_1 N_J(e_5, e_\alpha)_6, \nonumber\\
    N_J(e_1, e_\alpha)_5 &= -\mu_2 N_J(e_2, e_\alpha)_3 - \lambda_1 N_J(e_4, e_\alpha)_6,\nonumber\\
    N_J(e_1, e_\alpha)_6 &= \nu_2 N_J(e_2, e_\alpha)_3 + \lambda_1 N_J(e_4, e_\alpha)_5,\nonumber\\
    N_J(e_2, e_\alpha)_4 &= -\lambda_2 N_J(e_1, e_\alpha)_3 - \mu_1 N_J(e_5, e_\alpha)_6,\nonumber\\
    N_J(e_2, e_\alpha)_5 &= \mu_2 N_J(e_1, e_\alpha)_3 + \mu_1 N_J(e_4, e_\alpha)_6,\label{eq3.6}\\
    N_J(e_2, e_\alpha)_6 &= -\nu_2 N_J(e_1, e_\alpha)_3 - \mu_1 N_J(e_4, e_\alpha)_5,\nonumber\\
    N_J(e_3, e_\alpha)_4 &= \lambda_2 N_J(e_1, e_\alpha)_2 + \nu_1 N_J(e_5, e_\alpha)_6,\nonumber\\
    N_J(e_3, e_\alpha)_5 &= -\mu_2 N_J(e_1, e_\alpha)_2 - \nu_1 N_J(e_4, e_\alpha)_6,\nonumber\\
    N_J(e_3, e_\alpha)_6 &= \nu_2 N_J(e_1, e_\alpha)_2 + \nu_1 N_J(e_4, e_\alpha)_5,\nonumber
\end{align}
where $\alpha\in\{7, 8\}$. We check the first equality in \eqref{eq3.6} where $\alpha = 7$. By definition, we have
\[
N_J(e_1, e_7)_4  = \big\langle [e_1, e_7] + J[Je_1, e_7] + J[e_1, Je_7] - [Je_1, Je_7], e_4 \big\rangle.
\]
By a direct calculation, we see that
\begin{align*}
    \Big\langle [e_1, e_7], e_4 \Big\rangle = &0,\\
\Big\langle J[Je_1, e_7], e_4 \Big\rangle = &\lambda_2 (-c_{12} + \lambda_1^2c_{12}+\lambda_1\mu_1c_{11}+\lambda_1\nu_1b_{11})\\ &+ \lambda_1 (-t_{15} + \lambda_2^2t_{15} + \lambda_2\mu_2t_{14}+\lambda_2\nu_2s_{14})\\
= &\lambda_2\Big\langle [e_2, e_7] + J[Je_2, e_7], e_3 \Big\rangle + \lambda_1 \Big\langle [e_5, e_7] + J[Je_5, e_7], e_6 \Big\rangle,\\
\Big\langle J[e_1, Je_7], e_4 \Big\rangle
= &\lambda_2(-\mu_1b_{21}+\nu_1c_{21})\\=&\lambda_2\Big\langle J[e_2, Je_7] - [Je_2, Je_7], e_3 \Big\rangle,\\
-\Big\langle [Je_1, Je_7], e_4 \Big\rangle
= &\lambda_1(-\mu_2s_{24}+\nu_2t_{24})\\= &\lambda_1\Big\langle J[e_5, Je_7] - [Je_5, Je_7], e_6 \Big\rangle.
\end{align*}
Adding these equalities together gives us
\[
N_J(e_1, e_7)_4 = \lambda_2 N_J(e_2, e_7)_3 + \lambda_1 N_J(e_5, e_7)_6.
\]
The other equalities follow by a similar calculation. From \eqref{eq3.6} we see that the solutions to case (ii) also solve case (i). Hence, it suffices to consider case (ii).\smallskip

We notice that the system of equations of case (ii) can be written as
\[
  \left(\begin{array}{@{}cc@{}}
    P_{11} & 0 \\\
    0 & P_{22}
  \end{array}\right)\left(\begin{array}{@{}c@{}}
    \textbf{u} \\
    \textbf{v}
  \end{array}\right)=0,
\]
where $P_{11}, P_{22} \in \rn^{6 \times 6}$, $\textbf{u} = (b_{11}, b_{21}, c_{11}, c_{12}, c_{21}, c_{22})^T$,$\textbf{v} = (s_{14}, s_{24}, t_{14}, t_{15}, t_{24}, t_{25})^T$. We denote the system of equations $P_{11}\textbf{u} = 0$ by $\mathcal{S}_1$ and the system $P_{22}\textbf{v} = 0$ by $\mathcal{S}_2$.\smallskip

The system $\mathcal{S}_1$ includes the cases when $j, k\in\{1, 2, 3\}, j<k$ and $\alpha\in\{7, 8\}$.
\begin{equation}
  \left(\begin{array}{@{}cccccc@{}}
    0 & \nu_1^2-1 & \lambda_1 & -\mu_1 & \mu_1\nu_1 & \lambda_1\nu_1 \\
    -\nu_1^2+1 & 0 & -\mu_1\nu_1 & -\lambda_1\nu_1 & \lambda_1 & -\mu_1 \\
    -\lambda_1 & \mu_1\nu_1 & 0 & \nu_1 & \mu_1^2-1 & \lambda_1\mu_1 \\
    \mu_1 & \lambda_1\nu_1 & -\nu_1 & 0 & \lambda_1\mu_1 & \lambda_1^2-1 \\
    -\mu_1\nu_1 & -\lambda_1 & -\mu_1^2+1 & -\lambda_1\mu_1 & 0 & \nu_1 \\
    -\lambda_1\nu_1 & \mu_1 & -\lambda_1\mu_1 & -\lambda_1^2+1 & -\nu_1 & 0
  \end{array}\right)
   \left(\begin{array}{@{}c@{}}
    b_{11} \\
    b_{21} \\
    c_{11} \\
    c_{12} \\
    c_{21} \\
    c_{22}
  \end{array}\right) = \textbf{0}.\nonumber
\end{equation}
Since we have assumed that $\lambda_1 \neq 0$, the first row and the third row are linearly independent. Denote the $j$th row of the matrix as $\mathcal{R}_j$, then we have
\[
\frac{-\mu_1\nu_1}{\lambda_1}\cdot \mathcal{R}_1 + \frac{-\nu_1^2+1}{-\lambda_1}\cdot\mathcal{R}_3 = \mathcal{R}_2,
\]
\[
\frac{-\nu_1}{\lambda_1}\cdot\mathcal{R}_1 + \frac{-\mu_1}{\lambda_1}\cdot\mathcal{R}_3 = \mathcal{R}_4,
\]
\[
\frac{1-\mu_1^2}{\lambda_1}\cdot\mathcal{R}_1 + \frac{\mu_1\nu_1}{\lambda_1}\cdot\mathcal{R}_3 = \mathcal{R}_5,
\]
\[
-\mu_1\cdot\mathcal{R}_1 + \nu_1\cdot\mathcal{R}_3 = \mathcal{R}_6.
\]
So the solution of $\mathcal{S}_1$ is 4-dimensional, which is
\begin{align*}
    c_{11} &= \frac{\lambda_1^2+\mu_1^2}{\lambda_1}b_{21}+\frac{\mu_1}{\lambda_1}c_{12}-\frac{\mu_1\nu_1}{\lambda_1}c_{21}-\nu_1c_{22},\\
    b_{11} &= \frac{\mu_1\nu_1}{\lambda_1}b_{21} + \frac{\nu_1}{\lambda_1}c_{12} -\frac{\lambda_1^2+\nu_1^2}{\lambda_1}c_{21} + \mu_1c_{22},
\end{align*}
where $b_{21}, c_{12}, c_{21}, c_{22} \in \rn$.\smallskip

The system $\mathcal{S}_2$ includes the cases when $j, k\in\{4, 5, 6\}, j<k$, and $\alpha\in\{7, 8\}$.
\begin{equation}
  \left(\begin{array}{@{}cccccc@{}}
    0 & \nu_2^2-1 & \lambda_2 & -\mu_2 & \mu_2\nu_2 & \lambda_2\nu_2 \\
    -\nu_2^2+1 & 0 & -\mu_2\nu_2 & -\lambda_2\nu_2 & \lambda_2 & -\mu_2 \\
    -\lambda_2 & \mu_2\nu_2 & 0 & \nu_2 & \mu_2^2-1 & \lambda_2\mu_2 \\
    \mu_2 & \lambda_2\nu_2 & -\nu_2 & 0 & \lambda_2\mu_2 & \lambda_2^2-1 \\
    -\mu_2\nu_2 & -\lambda_2 & -\mu_2^2+1 & -\lambda_2\mu_2 & 0 & \nu_2 \\
    -\lambda_2\nu_2 & \mu_2 & -\lambda_2\mu_2 & -\lambda_2^2+1 & -\nu_2 & 0
  \end{array}\right)
   \left(\begin{array}{@{}c@{}}
    s_{14} \\
    s_{24} \\
    t_{14} \\
    t_{15} \\
    t_{24} \\
    t_{25}
  \end{array}\right) = \textbf{0}.\nonumber
\end{equation}
Similarly, since $\lambda_2 \neq 0$, $\mathcal{S}_1$ has the following 4-dimensional solution
\begin{align*}
    t_{14} &= \frac{\lambda_2^2+\mu_2^2}{\lambda_2}s_{24}+\frac{\mu_2}{\lambda_2}t_{15}-\frac{\mu_2\nu_2}{\lambda_2}t_{24}-\nu_2t_{25},\\
    s_{14} &= \frac{\mu_2\nu_2}{\lambda_2}s_{24} + \frac{\nu_2}{\lambda_2}t_{15} -\frac{\lambda_2^2+\nu_2^2}{\lambda_2}t_{24} + \mu_2t_{25},
\end{align*}
where $s_{24}, t_{15}, t_{24}, t_{25} \in \rn$.\smallskip

Since $\mathcal{S}_1$ and $\mathcal{S}_2$ are independent of each other, case (ii) thus the whole system of 72 equations has 8-dimensional solutions. Since in the structure equations of $\g$, $\rho$ can change freely in $\rn$, we have shown that for any fixed $J$ taking the form \eqref{matrixJ}, there exists exactly a 9-dimensional family of coefficients for $\g$ such that $(G, g, J)$ is a Hermitian manifold. The Lie bracket relations are given by \eqref{eq1} and the coefficients satisfy the following linear relations
\begin{align}
    c_{11} &= \frac{\lambda_1^2+\mu_1^2}{\lambda_1}b_{21}+\frac{\mu_1}{\lambda_1}c_{12}-\frac{\mu_1\nu_1}{\lambda_1}c_{21}-\nu_1c_{22},\nonumber\\
    b_{11} &= \frac{\mu_1\nu_1}{\lambda_1}b_{21} + \frac{\nu_1}{\lambda_1}c_{12} -\frac{\lambda_1^2+\nu_1^2}{\lambda_1}c_{21} + \mu_1c_{22},\nonumber\\
    t_{14} &= \frac{\lambda_2^2+\mu_2^2}{\lambda_2}s_{24}+\frac{\mu_2}{\lambda_2}t_{15}-\frac{\mu_2\nu_2}{\lambda_2}t_{24}-\nu_2t_{25},\nonumber\\
    s_{14} &= \frac{\mu_2\nu_2}{\lambda_2}s_{24} + \frac{\nu_2}{\lambda_2}t_{15} -\frac{\lambda_2^2+\nu_2^2}{\lambda_2}t_{24} + \mu_2t_{25},\nonumber
\end{align}
where $\rho, b_{21}, c_{12}, c_{21}, c_{22}, s_{24}, t_{15}, t_{24}, t_{25}\in \rn$.

\end{proof}

\chapter{Hermitian Semi-K\"ahler Manifolds}

In this chapter we look for the conditions when $G^8$ belongs to the class of Hermitian semi-K\"{a}hler manifolds i.e. $\mathcal{W}_3$. Since we have solved the integrability condition for $G^8$ in Chapter 3, it suffices to check when $\Tilde{\alpha}(Z) = 0$ for all tangent vectors $Z \in T_eG^8 \cong \g$. By the Koszul formula, for any $1 \leq j \leq 8$, we have
\begin{align*}
    &(\nabla_{e_j}\omega)(e_j, Z)\\ = &\;e_j(\omega(e_j, Z)) - \omega(\nabla_{e_j}e_j, Z) - \omega(e_j, \nabla_{e_j}Z)\\=&\;e_j(g(Je_j, Z)) - g(J\nabla_{e_j}e_j, Z) - g(Je_j, \nabla_{e_j}Z)\\=&\;g(\nabla_{e_j}Je_j, Z) + g(Je_j, \nabla_{e_j}Z) - g(J\nabla_{e_j}e_j, Z) - g(Je_j, \nabla_{e_j}Z)\\=&\;g(\nabla_{e_j}Je_{j} - J\nabla_{e_j}e_j, Z)\\
    = &\;g(\nabla_{e_j}Je_{j}, Z) + g(\nabla_{e_j}e_j, JZ)\\
    = &\;\frac{1}{2}\big(g([e_j, Je_j], Z) - g([e_j, Z], Je_j) - g([Je_j, Z], e_j) - 2g([e_j, JZ], e_j)\big),
\end{align*}
where we use the same notation for a vector in $\g$ and the left-invariant vector field generated by it since $\nabla \omega$ is tensorial. Hence for a basis vector $e_k$ of $\g$, we have
\[
\Tilde{\alpha}(e_k) = \frac{1}{2}\sum_{j=1}^8 \big(g([e_j, Je_j], e_k) - g([e_j, e_k], Je_j) - g([Je_j, e_k], e_j) - 2g([e_j, Je_k], e_j)\big),
\]
where $k \in \{1, \dots, 8\}$. By a standard calculation, we obtain the following results
\begin{align}
    \Tilde{\alpha}(e_1) &= -\frac{1}{2}(-\rho c_{12} + b_{11}c_{21} - b_{21}c_{11} - 4\lambda_1),\nonumber\\
    \Tilde{\alpha}(e_2) &= -\frac{1}{2}(\rho c_{11} + b_{11}c_{22} - b_{21}c_{12} + 4\mu_1),\nonumber\\
    \Tilde{\alpha}(e_3) &= -\frac{1}{2}(-\rho b_{11} + c_{11}c_{22} - c_{12}c_{21} - 4\nu_1),\nonumber\\
    \Tilde{\alpha}(e_4) &= -\frac{1}{2}(-\rho t_{15} + s_{14}t_{24} - s_{24}t_{14} - 4\lambda_2),\label{eq2}\\
    \Tilde{\alpha}(e_5) &= -\frac{1}{2}(\rho t_{14} + s_{14}t_{25} - s_{24}t_{15} + 4\mu_2),\nonumber\\
    \Tilde{\alpha}(e_6) &= -\frac{1}{2}(-\rho s_{14} + t_{14}t_{25} - t_{15}t_{24} - 4\nu_2),\nonumber\\
    \Tilde{\alpha}(e_7) &= 0,\nonumber\\
    \Tilde{\alpha}(e_8) &= 0\nonumber.
\end{align}
Recall that in \eqref{eq1}, the coefficients $\theta_j = \langle [X, Y], e_j\rangle$ are given by
\[
\left(\begin{array}{@{}c@{}}
\theta_1\\
\theta_2\\
\theta_3\\
\theta_4\\
\theta_5\\
\theta_6
  \end{array}\right)
=\frac{1}{2}
\left(\begin{array}{@{}c@{}}
-\rho c_{12} + b_{11}c_{21} - b_{21}c_{11}\\
\rho c_{11} + b_{11}c_{22} - b_{21}c_{12}\\
-\rho b_{11} + c_{11}c_{22} - c_{12}c_{21}\\
-\rho t_{15} + s_{14}t_{24} - s_{24}t_{14}\\
\rho t_{14} + s_{14}t_{25} - s_{24}t_{15}\\
-\rho s_{14} + t_{14}t_{25} - t_{15}t_{24}
  \end{array}\right).
\]
By comparing $\theta_j$ and \eqref{eq2}, we see that $G^8$ belongs to $\mathcal{W}_3$ if and only if
\begin{align}
    \theta_1 &= 2 \lambda_1, \theta_2 = -2 \mu_1, \theta_3 = 2 \nu_1,\nonumber\\
    \theta_4 &= 2 \lambda_2, \theta_5 = -2 \mu_2, \theta_6 = 2 \nu_2.\nonumber
\end{align}
Since $\lambda_j^2 + \mu_j^2 + \nu_j^2 = 1$, where $j = 1, 2$, at least two of $\theta_k$ must be nonzero. Therefore, for any $G^8$ in the class $\mathcal{W}_3$, we have
\[[X, Y] \notin \m,\]
i.e. the horizontal distribution is not integrable.\smallskip

Now we have the following system of equations for $G^8$ to be Hermitian semi-K\"{a}hler when $\lambda_1, \lambda_2 \neq 0$ and $\eta = 1$.
\begin{align}
-\rho c_{12} + b_{11}c_{21} - b_{21}c_{11} &= 4\lambda_1,\nonumber\\
    \rho c_{11} + b_{11}c_{22} - b_{21}c_{12} &= -4\mu_1,\nonumber\\
  -\rho b_{11} + c_{11}c_{22} - c_{12}c_{21} &= 4\nu_1,\nonumber\\
  (\lambda_1^2+\mu_1^2)b_{21}-\lambda_1 c_{11}+\mu_1 c_{12}-\mu_1\nu_1 c_{21}-\lambda_1\nu_1c_{22}&=0,\nonumber\\
    -\lambda_1 b_{11} + \mu_1\nu_1b_{21} + \nu_1 c_{12} -(\lambda_1^2+\nu_1^2)c_{21} + \lambda_1 \mu_1c_{22}&=0,\label{system of equations for W3}\\
    -\rho t_{15} + s_{14}t_{24} - s_{24}t_{14} &= 4\lambda_2,\nonumber\\
    \rho t_{14} + s_{14}t_{25} - s_{24}t_{15} &= -4\mu_2,\nonumber\\
    -\rho s_{14} + t_{14}t_{25} - t_{15}t_{24} &= 4\nu_2,\nonumber\\ (\lambda_2^2+\mu_2^2)s_{24} -\lambda_2 t_{14} +\mu_2 t_{15}-\mu_2\nu_2 t_{24}-\lambda_2 \nu_2t_{25}&=0,\nonumber\\
    -\lambda_2 s_{14} + \mu_2\nu_2 s_{24} + \nu_2 t_{15} -(\lambda_2^2+\nu_2^2)t_{24} + \lambda_2\mu_2t_{25} &= 0.\nonumber
\end{align}

\begin{theorem}
Let $G^8$ be a 8-dimensional Riemannian Lie group containing $K = \SU{2} \times \SU{2}$ as a Lie subgroup, where $K$ is endowed with the standard Riemannian metric induced by its Killing form. The Lie bracket relations are given by Proposition \ref{Turner equations}. Let $J$ be the left-invariant almost Hermitian structure on $G^8$ defined by $\eqref{matrixJ}$, where $\lambda_j^2 + \mu_j^2 + \nu_j^2 = 1$, $j = 1, 2$. Without loss of generality, we assume that $\lambda_1, \lambda_2 \neq 0$. Then $(G^8, g, J)$ belongs to the class $\mathcal{W}_3$ if and only if $\rho \neq 0$ and there exist $r_1, r_2 \in \rn$ such that
\[
    c_{12} = -\frac{4\lambda_1}{\rho},\;
    c_{11} = -\frac{4\mu_1}{\rho},\;
    b_{11} = -\frac{4\nu_1}{\rho},
    c_{22} = r_1\lambda_1,\;
    c_{21} = r_1\mu_1,\;
    b_{21} = r_1\nu_1,
\]
\[
    t_{15} = -\frac{4\lambda_2}{\rho},\;
    t_{14} = -\frac{4\mu_2}{\rho},\;
    s_{14} = -\frac{4\nu_2}{\rho},
    t_{25} = r_2\lambda_2,\;
    t_{24} = r_2\mu_2,\;
    s_{24} = r_2\nu_2.
\]
\label{theorem for W3}
\end{theorem}
\begin{remark}
By Theorem \ref{theorem for W3}, for a fixed matrix representation of the almost Hermitian structure $J$, there exists exactly a 3-dimensional family of the coefficients for $\g$ such that $G^8$ is Hermitian semi-K\"{a}hler. The space of solutions to the problem forms a fibre bundle over $\mathbb{S}^2 \times \mathbb{S}^2$ with fibre $\rn^* \times \rn^2$.
\end{remark}
\begin{proof}
Notice that for a fixed value of $\rho$, the equations in \eqref{system of equations for W3} that involves $b_{jk}, c_{jk}$ are independent of those involving $s_{jk}, t_{jk}$. Furthermore, these two subsystems have the same structure if we replace $b_{11}, b_{21}, c_{11}, c_{12}, c_{21}, c_{22}, \lambda_1, \mu_1, \nu_1$ by $s_{14}, s_{24}, t_{14}, t_{15}, t_{24},$

\noindent $ t_{25}, \lambda_2, \mu_2, \nu_2$, respectively. So it suffices to solve the system of the equations about $b_{jk}, c_{jk}$. The subsystem of equations is given by
\begin{align*}
    -\rho c_{12} + b_{11}c_{21} - b_{21}c_{11} &= 4\lambda_1,\\
    \rho c_{11} + b_{11}c_{22} - b_{21}c_{12} &= -4\mu_1,\\
  -\rho b_{11} + c_{11}c_{22} - c_{12}c_{21} &= 4\nu_1,\\
  (\lambda_1^2+\mu_1^2)b_{21}-\lambda_1 c_{11}+\mu_1 c_{12}-\mu_1\nu_1 c_{21}-\lambda_1\nu_1c_{22}&=0,\\
    -\lambda_1 b_{11} + \mu_1\nu_1b_{21} + \nu_1 c_{12} -(\lambda_1^2+\nu_1^2)c_{21} + \lambda_1 \mu_1c_{22}&=0.
\end{align*}
Denote $(c_{12}, -c_{11}, b_{11})^\top$, $(c_{22}, -c_{21}, b_{21})^\top$, $(\lambda_1, -\mu_1, \nu_1)^\top$ by $\mathbf{u}$, $\mathbf{v}$, $\mathbf{w}$, respectively. Then the system of equations can be rewritten as
\begin{equation}
-\rho \mathbf{u} + \mathbf{u}\times\mathbf{v} = 4\mathbf{w},\label{W3main}
\end{equation}
\[
\begin{pmatrix}\mu_1 \\ \lambda_1 \\ 0\end{pmatrix}\cdot\mathbf{u} + \begin{pmatrix}-\lambda_1 \nu_1 \\ \mu_1\nu_1 \\ \lambda_1^2 + \mu_1^2\end{pmatrix}\cdot\mathbf{v} = 0,
\]
\[
\begin{pmatrix}\nu_1 \\ 0 \\ -\lambda_1\end{pmatrix}\cdot\mathbf{u} + \begin{pmatrix}\lambda_1 \mu_1 \\ \lambda_1^2 + \nu_1^2 \\ \mu_1 \nu_1\end{pmatrix}\cdot\mathbf{v} = 0.
\]
Substitute $\mathbf{u}, \mathbf{v}$ by $(x_1, y_1, z_1)$, $(x_2, y_2, z_2)$, respectively. Then the system of the last two equations is equivalent to
\begin{align*}
\mu_1 x_1 + \lambda_1 y_1 &= C_1,\\
\nu_1 x_1 - \lambda_1 z_1 &= C_2,\\
-\lambda_1 \nu_1 x_2 + \mu_1\nu_1 y_2 + (\lambda_1^2 + \mu_1^2) z_2 &= -C_1,\\
\lambda_1 \mu_1 x_2 + (\lambda_1^2 + \nu_1^2) y_2 + \mu_1 \nu_1 z_2 &= -C_2, 
\end{align*}
where $C_1, C_2 \in \rn$. 
The general solution to the first two equations is
\begin{equation}
	\begin{pmatrix}
		x_1\\
		y_1\\
		z_1
	\end{pmatrix}
	= l_1
	\begin{pmatrix}
		\lambda_1\\
		-\mu_1\\
		\nu_1
	\end{pmatrix} = l_1 \mathbf{w},\label{general_solution_1}
\end{equation}
where $l_1 \in \rn$. By substitution, we see that this is also the general solution to the last two equations, i.e.
\begin{equation}
	\begin{pmatrix}
		x_2\\
		y_2\\
		z_2
	\end{pmatrix}
	= r_1
	\begin{pmatrix}
		\lambda_1\\
		-\mu_1\\
		\nu_1
	\end{pmatrix} = r_1 \mathbf{w},\label{general_solution_2}
\end{equation}
where $r_1 \in \rn$. Therefore, we have
\begin{align*}
    \mathbf{u} &= \mathbf{u}_0 + l_1 \mathbf{w},\\
    \mathbf{v} &= \mathbf{v}_0 + r_1 \mathbf{w},
\end{align*}
where $l_1, r_1 \in \rn$ and $\mathbf{u}_0, \mathbf{v}_0$ are special solutions to the system for fixed $C_1, C_2$. Especially, we can choose $\mathbf{u}_0, \mathbf{v}_0$ such that they are orthogonal to $\mathbf{w}$, i.e. $\mathbf{u}_0$ is the solution to
\begin{align*}
    \mu_1 x_1 + \lambda_1 y_1 &= C_1,\\
    \nu_1 x_1 - \lambda_1 z_1 &= C_2,\\
    \lambda_1 x_1 - \mu_1 y_1 + \nu_1 z_1 &= 0,
\end{align*}
and $\mathbf{v}_0$ is the solution to
\begin{align*}
    -\lambda_1 \nu_1 x_2 + \mu_1\nu_1 y_2 + (\lambda_1^2 + \mu_1^2) z_2 &= -C_1,\\
    \lambda_1 \mu_1 x_2 + (\lambda_1^2 + \nu_1^2) y_2 + \mu_1 \nu_1 z_2 &= -C_2,\\
    \lambda_1 x_2 - \mu_1 y_2 + \nu_1 z_2 &= 0.
\end{align*}
Solving the two systems gives us
\begin{equation}
    \mathbf{u}_0 = \begin{pmatrix}
    \nu_1 C_1 + \mu_1 C_2\\[0.3em]
    -\frac{1}{\lambda_1}(\mu_1 \nu_1 C_1 + \mu_1^2 C_2 - C_2)\\[0.3em]
    \frac{1}{\lambda_1}(\nu_1^2 C_1 - C_1 + \mu_1 \nu_1 C_2)
    \end{pmatrix}
\end{equation}
and
\begin{equation}
    \mathbf{v}_0 = \begin{pmatrix}
    \frac{1}{\lambda_1}(-\mu_1 C_1 + \nu_1 C_2)\\[0.3em]
    -C_1\\[0.3em]
    -C_2
    \end{pmatrix}.\label{norm_v0}
\end{equation}
By a direct calculation, we see that
\[
\mathbf{u}_0 \cdot \mathbf{v}_0 = 0.
\]
Hence, $\mathbf{u}_0, \mathbf{v}_0, \mathbf{w}$ are pairwise orthogonal to each other. Furthermore, we have
\[
\mathbf{w} \times \mathbf{v}_0 = 
\begin{vmatrix}
\mathbf{e}_1 & \mathbf{e}_2 & \mathbf{e}_3 \\
\lambda_1    & -\mu_1       & \nu_1\\
 \frac{1}{\lambda_1}(-\mu_1 C_1 + \nu_1 C_2)
 & -C_1 & -C_2
\end{vmatrix} = \mathbf{u}_0.
\]
Since $\mathbf{w}$ and $\mathbf{v}_0$ are orthogonal and $\lVert\mathbf{w}\rVert = 1$, we have
\[
\lVert\mathbf{u}_0\rVert = \lVert\mathbf{v}_0\rVert.
\]
Therefore, the basis $[\mathbf{w}, \mathbf{v}_0, \mathbf{u}_0]$ is positively oriented and
\[
\mathbf{u}_0 \times \mathbf{w} = \mathbf{v}_0. 
\]
Moreover, since $\lVert\mathbf{u}_0\rVert = \lVert\mathbf{v}_0\rVert$ and $\lVert\mathbf{w}\rVert = 1$, we have
\[
    \mathbf{u}_0 \times \mathbf{v}_0 = -\lVert\mathbf{v}_0\rVert^2 \;\mathbf{w}.
\]
Using these results, for \eqref{W3main}, we have
\begin{align}
&-\rho \mathbf{u} + \mathbf{u}\times\mathbf{v} = 4\mathbf{w}\nonumber\\ \iff &-\rho (\mathbf{u}_0 + l_1 \mathbf{w}) + (\mathbf{u}_0 + l_1 \mathbf{w})\times(\mathbf{v}_0 + r_1 \mathbf{w}) = 4\mathbf{w}\nonumber\\
\iff &-\rho (\mathbf{u}_0 + l_1 \mathbf{w}) + \mathbf{u}_0 \times \mathbf{v}_0 + r_1 \mathbf{u}_0 \times \mathbf{w} + l_1
\mathbf{w} \times \mathbf{v}_0 = 4 \mathbf{w}\nonumber\\
\iff &-\rho (\mathbf{u}_0 + l_1 \mathbf{w}) -\lVert\mathbf{v}_0\rVert^2 \;\mathbf{w} + r_1 \mathbf{v}_0 + l_1 \mathbf{u}_0 = 4 \mathbf{w}\nonumber\\
\iff \;&(4 + \rho l_1 +\lVert\mathbf{v}_0\rVert^2) \mathbf{w} - r_1 \mathbf{v}_0 + (\rho - l_1)\mathbf{u}_0 = 0.\label{simplified_main_W3}
\end{align}
When $\lVert\mathbf{v}_0\rVert \neq 0$, the vectors $\mathbf{w}, \mathbf{v}_0, \mathbf{u}_0$ form a basis for $\rn^3$. Hence, the equation is equivalent to
\begin{align*}
    4 + \rho l_1 + \lVert\mathbf{v}_0\rVert^2 &= 0,\\
    -r_1 &= 0,\\
    \rho - l_1 &= 0.
\end{align*}
Substituting the third equation $l_1 = \rho$ into the first one, we see that
\[
4 + \rho^2 + \lVert\mathbf{v}_0\rVert^2 > 0.
\]
Hence, the system has no solution when $\lVert\mathbf{v}_0\rVert \neq 0$. It suffices to consider the case when $\lVert\mathbf{v}_0\rVert = 0$. By \eqref{norm_v0}, this is equivalent to
\[
C_1 = C_2 = 0.
\]
In this situation, the system is solved by the general solutions \eqref{general_solution_1} and \eqref{general_solution_2}. By \eqref{simplified_main_W3}, we have
\begin{align*}
    -\rho \mathbf{u} + \mathbf{u}\times\mathbf{v} = 4\mathbf{w} &\iff (4 + \rho l_1)\mathbf{w} = 0\\
    &\iff 4 + \rho l_1 = 0.
\end{align*}
Therefore, the system of equations has solutions if and only if $\rho \neq 0$, and the solutions are given by setting $l_1 = -\frac{4}{\rho}$ in the general solution \eqref{general_solution_1}, i.e.
\[
    x_1 = -\frac{4\lambda_1}{\rho},\;
    y_1 = \frac{4\mu_1}{\rho},\;
    z_1 = -\frac{4\nu_1}{\rho},
    x_2 = r_1\lambda_1,\;
    y_2 = -r_1\mu_1,\;
    z_2 = r_1\nu_1,
\]
where $\rho \in \rn^*$ and $r_1 \in \rn$.\smallskip

As we have addressed before, the same argument can be applied to the other part of the system which involves $s_{jk}, t_{lm}$. Therefore, by the above discussion, we have proven that $G^8$ belongs to $\mathcal{W}_3$ if and only if
\[
    c_{12} = -\frac{4\lambda_1}{\rho},\;
    c_{11} = -\frac{4\mu_1}{\rho},\;
    b_{11} = -\frac{4\nu_1}{\rho},
    c_{22} = r_1\lambda_1,\;
    c_{21} = r_1\mu_1,\;
    b_{21} = r_1\nu_1,
\]
\[
    t_{15} = -\frac{4\lambda_2}{\rho},\;
    t_{14} = -\frac{4\mu_2}{\rho},\;
    s_{14} = -\frac{4\nu_2}{\rho},
    t_{25} = r_2\lambda_2,\;
    t_{24} = r_2\mu_2,\;
    s_{24} = r_2\nu_2,
\]
where $\rho \in \rn^*$, $r_1, r_2 \in \rn$. Thus the solution to this problem is 3-dimensional.

\end{proof}

\chapter{Locally Conformal K\"ahler Manifolds}

In this chapter, we investigate the conditions when $G^8$ belongs to the class $\mathcal{W}_4$. By the tensoriality of $\mu$ and the left translations on $G^8$, it suffices to check the condition for any basis vectors $e_j, e_k, e_l$ in the orthonormal basis of $\g$.\smallskip

We define the function $\psi$ by
\begin{align*}
    \psi(j, k, l) = \,&\alpha(e_j, e_k, e_l) + \frac{1}{2(n - 1)}(g(e_j, e_k)\Tilde{\alpha}(e_l)-g(e_j, e_l)\Tilde{\alpha}(e_k)\\&-g(e_j, Je_k)\Tilde{\alpha}(Je_l)+g(e_j, Je_l)\Tilde{\alpha}(Je_k)).
\end{align*}
The Hermitian manifold $G^8$ belongs to $\mathcal{W}_4$ if and only if $\psi(j, k, l)$ vanishes for any $1 \leq j, k, l \leq 8$. By a direct calculation, we have:
\begin{align*}
    \psi(8, 8, 1) &= \frac{1}{3}\theta_1 + \frac{\lambda_1}{3},\,
    \psi(8, 8, 2) = \frac{1}{3}\theta_2 - \frac{\mu_1}{3},\,
    \psi(8, 8, 3) = \frac{1}{3}\theta_3 + \frac{\nu_1}{3},\\
    \psi(8, 8, 4) &= \frac{1}{3}\theta_4 + \frac{\lambda_2}{3},\,
    \psi(8, 8, 5) = \frac{1}{3}\theta_5 - \frac{\mu_2}{3},\,
    \psi(8, 8, 6) = \frac{1}{3}\theta_6 + \frac{\nu_2}{3}.
\end{align*}
Therefore, if $G^8$ belongs to $\mathcal{W}_4$, we must have
\begin{equation}
    \theta_1 = -\lambda_1,\, \theta_2 = \mu_1,\, \theta_3 = -\nu_1,\, \theta_4 = -\lambda_2,\, \theta_5 = \mu_2,\, \theta_6 = -\nu_2.\label{W4eq}
\end{equation}
Substituting $\theta_j$ by \eqref{W4eq}, we have
\[
\psi(1, 1, 4) = \frac{1}{2}\lambda_2(1 - \lambda_1^2),\,\,\psi(1, 1, 5) = \frac{1}{2}\mu_2(\lambda_1^2 - 1),\,\,\psi(1, 1, 6) = \frac{1}{2}\nu_2(1 - \lambda_1^2),
\]
\[
\psi(4, 4, 1) = \frac{1}{2}\lambda_1(1 - \lambda_2^2),\,\,\psi(4, 4, 2) = \frac{1}{2}\mu_1(\lambda_2^2 - 1),\,\,\psi(4, 4, 3) = \frac{1}{2}\nu_1(1 - \lambda_2^2).
\]
Since $\lambda_j^2 + \mu_j^2 + \nu_j^2 = 1$, we must have
\[
\lambda_1^2 = 1, \,\, \lambda_2^2 = 1, \,\,
\mu_1 = \nu_1 = 0,\,\,
\mu_2 = \nu_2 = 0.
\]
However, substituting $\mu_j, \nu_j$ by 0, we have
\[
\psi(2, 3, 4) = \frac{\lambda_1 \lambda_2}{2} \neq 0.
\]
Hence, $G^8$ never belongs to $\mathcal{W}_4$.

\appendix
\chapter{The Maple Programme for $G^8$}

In this appendix, we give a brief account of the Maple programme that we use to check our computations. First, we set up the environment.

\begin{verbatim}
	restart;
	with(DifferentialGeometry);
	with(LieAlgebras);
	with(LinearAlgebra);
	with(ArrayTools);
\end{verbatim}	

Then we import the structure equations of $G^8$ from the paper \cite{Tur} to the program.

\begin{verbatim}
	BasisVectors := [A, B, C, R, S, T, X, Y]:
	
	StructureEquations := [
	[A, B] = 2*C, [C, A] = 2*B, [B, C] = 2*A,
	[R, S] = 2*T, [T, R] = 2*S, [S, T] = 2*R,
	[A, X] = -B*b[11] - C*c[11], 
	[A, Y] = -B*b[21] - C*c[21], 
	[B, X] = A*b[11] - C*c[12], 
	[B, Y] = A*b[21] - C*c[22], 
	[C, X] = A*c[11] + B*c[12], 
	[C, Y] = A*c[21] + B*c[22], 
	[R, X] = -S*s[14] - T*t[14], 
	[R, Y] = -S*s[24] - T*t[24], 
	[S, X] = R*s[14] - T*t[15], 
	[S, Y] = R*s[24] - T*t[25], 
	[T, X] = R*t[14] + S*t[15], 
	[T, Y] = R*t[24] + S*t[25], 
	[X, Y] = rho*X 
	         + (1/2)*(-rho*c[12] + b[11]*c[21] - b[21]*c[11])*A 
	         + (1/2)*(rho*c[11] + b[11]*c[22] - b[21]*c[12])*B 
	         + (1/2)*(-rho*b[11] + c[11]*c[22] - c[12]*c[21])*C 
	         + (1/2*(-rho*t[15] + s[14]*t[24] - s[24]*t[14])*R 
	         + (1/2)*(rho*t[14] + s[14]*t[25] - s[24]*t[15])*S 
	         + (1/2)*(-rho*s[14] + t[14]*t[25] - t[15]*t[24])*T]:
\end{verbatim}	

Initialize the equations for checking the integrability. The parameters $$\lambda_1, \mu_1, \nu_1, \lambda_2, \mu_2, \nu_2$$ are replaced by $$x_1, y_1, z_1, x_2, y_2, z_2$$ for convenience.

\begin{verbatim}
	LAD := LieAlgebraData(StructureEquations, BasisVectors, alg):
	DGsetup(LAD):
	N := numelems(BasisVectors);
	ProcBasis := [cat(e, 1 .. N)];
	NJ := array(1 .. 8, 1 .. 8, 1 .. 8);
	Equs := {x1^2 + y1^2 + z1^2 = 1, x2^2 + y2^2 + z2^2 = 1};
	
\end{verbatim}

The function con() converts a data from the form of a vector (in LinearAlgebra) to a vector (in LieAlgebra) in $\g$.

\begin{verbatim}
	con := proc(V) 
	  local a, i; 
	  a := 0; 
	  for i to N do 
	    a := a + V[i]*ProcBasis[i];
	  end do; 
	end:
\end{verbatim}

The following procedure computes the Nijenhuis tensor for given vectors u, v and a given almost complex structure $J$.

\begin{verbatim}
	Nijenhuis := proc(u, v, index_u, index_v) 
	  local i, result, Ju, Jv, U, V, Juv, uJv, JJuv, JuJv; 
	  global Equs, NJ; 
	  result := 0; 
	  U := Vector(GetComponents(u, ProcBasis)); 
	  V := Vector(GetComponents(v, ProcBasis)); 
	  result := result + LieBracket(u, v); 
	  Ju := con(MatrixVectorMultiply(J, U)); 
	  Jv := con(MatrixVectorMultiply(J, V)); 
	  result := result - LieBracket(Ju, Jv); 
	  Juv := Vector(GetComponents(LieBracket(Ju, v), ProcBasis)); 
	  JJuv := con(MatrixVectorMultiply(J, Juv)); 
	  uJv := Vector(GetComponents(LieBracket(u, Jv), ProcBasis)); 
	  JuJv := con(MatrixVectorMultiply(J, uJv)); 
	  result := result + JJuv + JuJv; 
	  if (result != 0) then 
	    for i to N do 
	      if (GetComponents(result, ProcBasis)[i] != 0) then 
	        print(u, v, i); 
	        print(GetComponents(result, ProcBasis)[i]); 
	        Equs := Equs union {GetComponents(result, ProcBasis)[i] = 0}; 
	        NJ[index_u, index_v, i] := GetComponents(result, ProcBasis)[i]; 
	        end if; 
	    end do; 
	  end if; 
	end proc:
\end{verbatim}

The procedure CheckIntegrability() lists the result of Nijenhuis() and we use the output to analyze the symmetry in the Nijenhuis tensor.

\begin{verbatim}
	CheckIntegrability := proc() 
	  local i, j; 
	  for i to N - 1 do 
	    for j from i + 1 to N do 
	      Nijenhuis(ProcBasis[i], ProcBasis[j], i, j); 
	    end do; 
	  end do; 
	end proc:
\end{verbatim}

We import the result by Magnin to our program and run the procedure CheckIntegrability().

\begin{verbatim}
	xi := 0;
	eta := 1;
	J := <
	<0, -z1, -y1, -x1*x2, x1*y2, -x1*z2, 0, 0> |
	<z1, 0, -x1, y1*x2, -y1*y2, y1*z2, 0, 0>   | 
	<y1, x1, 0, -z1*x2, z1*y2, -z1*z2, 0, 0>   | 
	<x1*x2, -y1*x2, z1*x2, 0, -z2, -y2, 0, 0>  | 
	<-x1*y2, y1*y2, -z1*y2, z2, 0, -x2, 0, 0>  | 
	<x1*z2, -y1*z2, z1*z2, y2, x2, 0, 0, 0>    | 
	<0, 0, 0, 0, 0, 0, 0, 1> | 
	<0, 0, 0, 0, 0, 0, -1, 0>>:
	CheckIntegrability();
\end{verbatim}

Now we define a function Nabla() corresponding to $\alpha(e_j, e_k, e_l)$ in the thesis.

\begin{verbatim}
	Nabla := proc(index_u, index_v, index_w) 
	  local result, u, v, w, U, V, W, JU, JV, JW, Ju, Jv, Jw, UJV, UW,
	  JVW, UV, UJW, VJW; 
	  global Equs, NJ; 
	  result := 0; 
	  u := ProcBasis[index_u]; 
	  v := ProcBasis[index_v]; 
	  w := ProcBasis[index_w]; 
	  U := Vector(GetComponents(u, ProcBasis));
	  V := Vector(GetComponents(v, ProcBasis)); 
	  W := Vector(GetComponents(w, ProcBasis)); 
	  JU := MatrixVectorMultiply(J, U); 
	  JV := MatrixVectorMultiply(J, V); 
	  JW := MatrixVectorMultiply(J, W); 
	  Ju := con(JU); 
	  Jv := con(JV); 
	  Jw := con(JW); 
	  UJV := Vector(GetComponents(LieBracket(u, Jv), ProcBasis)); 
	  UW := Vector(GetComponents(LieBracket(u, w), ProcBasis)); 
	  JVW := Vector(GetComponents(LieBracket(Jv, w), ProcBasis)); 
	  UV := Vector(GetComponents(LieBracket(u, v), ProcBasis)); 
	  UJW := Vector(GetComponents(LieBracket(u, Jw), ProcBasis)); 
	  VJW := Vector(GetComponents(LieBracket(v, Jw), ProcBasis)); 
	  result := 1/2*result 
	  + 1/2*DotProduct(UJV, W, conjugate = false) 
	  - 1/2*DotProduct(UW, JV, conjugate = false) 
	  - 1/2*DotProduct(JVW, U, conjugate = false) 
	  + 1/2*DotProduct(UV, JW, conjugate = false) 
	  - 1/2*DotProduct(UJW, V, conjugate = false) 
	  - 1/2*DotProduct(VJW, U, conjugate = false); 
	  return result; 
	end:
\end{verbatim}

The function NablaOmegajjk() calculates $\nabla_{e_j} \omega(e_j, e_k)$.

\begin{verbatim}
	NablaOmegajjk := proc(u, v) 
	  local result, U, V, JU, JV, Ju, Jv, UJV, UJU, JUV, UV; 
	  result := 0; 
	  U := Vector(GetComponents(u, ProcBasis)); 
	  V := Vector(GetComponents(v, ProcBasis)); 
	  JU := MatrixVectorMultiply(J, U); 
	  JV := MatrixVectorMultiply(J, V); 
	  Ju := con(JU); Jv := con(JV); 
	  UJV := Vector(GetComponents(LieBracket(u, Jv), ProcBasis)); 
	  UJU := Vector(GetComponents(LieBracket(u, Ju), ProcBasis)); 
	  JUV := Vector(GetComponents(LieBracket(Ju, v), ProcBasis)); 
	  UV := Vector(GetComponents(LieBracket(u, v), ProcBasis)); 
	  result := (1/2)*DotProduct(UJU, V, conjugate = false) 
	            - (1/2)*DotProduct(UV, JU, conjugate = false) 
	            - (1/2)*DotProduct(JUV, U, conjugate = false) 
	            - DotProduct(UJV, U, conjugate = false); 
	  return result; 
	end:
\end{verbatim}

Then we can define $\bar{\alpha}(e_k)$ for basis vectors $e_k$ in the thesis.

\begin{verbatim}
	BarAlphaek := proc(v) 
	  local j, result; 
	  result := 0; 
	  for j to N do 
	    result := result + NablaOmegajjk(ProcBasis[j], v); 
	  end do; 
	  return result; 
	end:
\end{verbatim}

And we can calculate $\bar{\alpha}(v)$ for any $v \in \g$.

\begin{verbatim}
	BarAlpha := proc(v) 
	  local j, V, result; 
	  global N; 
	  result := 0; 
	  V := Vector(GetComponents(v, ProcBasis)); 
	  for j to N do 
	    result := result + V[j]*BarAlphaek(ProcBasis[j]); 
	  end do; 
	  return result; 
	end:
\end{verbatim}

The function $\psi(j, k, l)$ in Chapter 5 in the thesis is calculated by the following programme.

\begin{verbatim}
	W4 := proc(index_u, index_v, index_w) 
	  local result, u, v, w, U, V, W, JV, JW, Jv, Jw; 
	  global Equs, NJ, N; 
	  result := 0; 
	  u := ProcBasis[index_u]; 
	  v := ProcBasis[index_v]; 
	  w := ProcBasis[index_w]; 
	  U := Vector(GetComponents(u, ProcBasis)); 
	  V := Vector(GetComponents(v, ProcBasis)); 
	  W := Vector(GetComponents(w, ProcBasis)); 
	  JV := MatrixVectorMultiply(J, V); 
	  JW := MatrixVectorMultiply(J, W); 
	  Jv := con(JV); 
	  Jw := con(JW); 
	  result := Nabla(index_u, index_v, index_w) 
	  + (DotProduct(U, V, conjugate = false)*BarAlpha(w) 
	  - DotProduct(U, W, conjugate = false)*BarAlpha(v) 
	  - DotProduct(U, JV, conjugate = false)*BarAlpha(Jw) 
	  + DotProduct(U, JW, conjugate = false)*BarAlpha(Jv))/(N - 2);
	  return result; 
	end:
\end{verbatim} 

Finally, the following procedure gives the condition for $G^8 \in \mathcal{W}_4$.

\begin{verbatim}
	CheckW4 := proc()
	  local j, k, l; 
	  for j to N do 
	    for k to N do 
	      for l to N do 
	        print(j, k, l); 
	        print(W4(j, k, l)); 
	      end do; 
	    end do; 
	  end do; 
	end:
\end{verbatim}

\backcover
\end{document}